\documentclass[11pt]{elsarticle}

\usepackage{pdflscape}
\usepackage{multirow}

\usepackage{ amsmath, amssymb,mathrsfs,bm,bbm, amsfonts,setspace, amsthm}
   \usepackage{rotating}
\usepackage[numbers]{natbib}    
   \usepackage{etoolbox}

\usepackage{graphicx}
\usepackage[all]{xy}

\newtheorem{theorem}{Theorem}[section]
\newtheorem{remark}{Remark}[section]

\newtheorem{corollary}{Corollary}[section]
\newtheorem{definition}{Definition}[section]
\newtheorem{example}{Example}[section]

\numberwithin{equation}{section}

\def\argmax{\operatorname{argmax}}
\def\argmin{\operatorname{argmin}}
\def\e{\operatorname{e}}

\makeatletter
\def\@author#1{\g@addto@macro\elsauthors{\normalsize%
    \def\baselinestretch{1}%
    \upshape\authorsep#1\unskip\textsuperscript{%
      \ifx\@fnmark\@empty\else\unskip\sep\@fnmark\let\sep=,\fi
      \ifx\@corref\@empty\else\unskip\sep\@corref\let\sep=,\fi
      }%
    \def\authorsep{\unskip,\space}%
    \global\let\@fnmark\@empty
    \global\let\@corref\@empty  
    \global\let\sep\@empty}%
    \@eadauthor={#1}
}
\makeatother

\makeatletter
\patchcmd{\tnotemark}{\ding{73}}{\dag}{}{\@latex@error{Failed to path \string\tnotemarkspace for \string\ding{73}}}
 \patchcmd{\tnotemark}{\ding{73}\ding{73}}{\dag\dag}{}{\@latex@error{Failed to path \string\tnotemark\space for \string\ding{73}\string\ding{73}}}
  \patchcmd{\tnotetext}{\ding{73}}{\dag}{}{\@latex@error{Failed to path \string\tnotetext\space for \string\ding{73}}}

    \patchcmd{\tnotetext}{\ding{73}\ding{73}}{\dag\dag}{}{\@latex@error{Failed to path \string\tnotetext\space for \string\ding{73}\string\ding{73}}}
\makeatother

\begin{document}

\begin{frontmatter}

\title{Moments of order statistics   from DNID discrete random variables with application in reliability}


\author[ja]{Anna Dembi\'nska\corref{cor1}}
\ead{dembinsk@mini.pw.edu.pl}
 \cortext[cor1]{Corresponding author.}
 
 \author[Agnieszka]{Agnieszka Goroncy}
\ead{gemini@mat.umk.pl}

\address[ja]{Faculty of Mathematics and Information Science\\
Warsaw University of Technology\\
ul. Koszykowa 75, 00-662 Warsaw, Poland}

\address[Agnieszka]{Faculty of Mathematics and Computer Science\\
Nicolaus Copernicus University\\
ul. Chopina 12/18, 87-100 Toru\'n, Poland}

\begin{abstract}
In this paper, we present methods of obtaining single moments of order statistics arising  from posibly dependent and non-identically distributed discrete random variables. We derive exact and approximate formulas  convenient for numerical evaluation of such moments.  To demonstrate their use, we  tabulate means and second moments of order statistics from random vectors having some typical discrete distributions with selected parameter values.
Next, we apply our results in reliability theory. We establish moments of  discrete  lifetimes of  coherent systems consisting of heterogeneous and  not necessarily independently working components. In particular, we obtain expression for expectations and variances of lifetimes of such systems. We give some illustrative examples.

\end{abstract}


\end{frontmatter}

\section{Introduction}
\label{section1}
Let  $X_1,\ldots,X_n$ be possibly dependent and not necessarily identically distributed (DNID) discrete random variables (RV's) and by $X_{1:n},\ldots,X_{n:n}$ denote the corresponding order statistics. The  aim of this paper is to derive expressions for single  moments of $X_{r:n}$, $1\leq r\leq n$, and to use these results to obtain moments of lifetimes of 
 coherent systems consisting of components having discrete operating times.

The study of order statistics from discrete populations and their moments has a long history. As described in a review paper by Nagaraja \cite{N92} the interest in this topic arose from practical problems like  statistical testing for homogeneity, and ranking populations and selection of the best one. Tank and Eryilmaz \cite{TE15} pointed out  its  utility in   statistical control of quality and start-up demonstration test. Eisenberg \cite{E08} motivated computing the mean of the maximum of independent and identically (IID) geometric RV's by a statistical problem in bioinformatics. Another field where discrete order statistics find application is the reliability theory. They are  used in analysis of technical systems when component lifetimes have discrete probability distributions. Discrete component lifetimes appear naturally when the system is monitored in discrete time or when it performs a task repetitively and its components have certain probabilities of failure upon each cycle, see, e.g.  \cite[p 17]{BP96}. Some other examples where discrete lifetimes are involved are discrete-time shock models; for details see \cite{E16szok} and the references therein.

An important class of systems studied in the reliability theory is one consisting of so called coherent systems i.e. of systems  such that  every their component is relevant and  replacing a failed component by a working one cannot cause a working system to fail. There is a vast literature on coherent systems yet it is restricted mainly to the case when component lifetimes are jointly continuously distributed. A few results are known to be valid also in the most general case of arbitrary distributions of component lifetimes (see for instance \cite{NRS07}, \cite{NSBB08}, \cite{MR14},  \cite{BE15}, \cite{EKT16}) and these  can be applied in particular when components lifetimes are discrete. Results holding specifically for the discrete case are sparse and, to the best of our knowledge, concern only a~special subclass of coherent systems, namely $k$-out-of-$n$ systems (\cite{W62}, \cite{TE15}, \cite{D18}, \cite{DD19}); for the definition of $k$-out-of-$n$ systems see the beginning of Section~\ref{section4}. Consequently, there are still many open problems related to reliability properties of coherent systems with discrete lifetimes of components. This work is intended to solve one of such problems and to provide methods, convenient for numerical computations, of finding means, variances and other moments of times to failures of coherent systems operating in discrete time. For this purpose we first concentrate on single moments of discrete order statistics and derive exact and approximate formulas that allow their computation using software.

The paper is organized as follows. In Section \ref{section2}, we establish  expressions for single moments of order statistics from DNID discrete RV's. In the case when these expressions involve infinite summations, we propose a~procedure that can be used for numerical computations and leads to approximate results with an error not exceeding a given value. In particular, we present explicit formulas for obtaining desired approximations when the marginal distributions of the parent RV's are Poisson or negative binomial. In Section  \ref{section3}, we look at samples with the multivariate geometric distribution  introduced in \cite{EM73} to describe dependent  lifetimes of  units exposed to common shocks. We derive closed-form expressions which enable numerical computation of exact values of moments of the corresponding order statistics. Section \ref{section4} is devoted to applications of the results given in two preceding  sections in the reliability theory to find moments of the lifetime  of a coherent system operating in discrete time.  We provide general exact and approximate formulas for these  moments and evaluate them for the bridge system with some selected joint distributions of component lifetimes. The results presented in the paper are illustrated in numerous tables and some figures.

\section{Moments of order statistics from DNID discrete RV's}
\label{section2}

In \cite{DD18}  expressions for moments of order statistics arising from independent and not necessarily identically distributed (INID) discrete RV's were derived. We begin this section with a generalization of this result to the case when the underlying RV's are DNID.

\begin{theorem}
\label{th1}
Let $X_1,X_2,\ldots,X_n$ be DNID RV's taking values in the set of non-negative integers. Then, for $1\leq r\leq n$ and $p=1,2,\ldots$, we have
\begin{eqnarray}
EX^p_{r:n}&=&\sum\limits_{m=0}^\infty \big((m+1)^p-m^p\big)  \sum_{s=0}^{r-1}  \sum_{(j_1,\ldots,j_n)\in\mathcal{P}_s}   \nonumber \\
&&
P\left(\bigg(\bigcap\limits_{l=1}^s\left\{X_{j_l}\leq m\right\}\bigg)\cap \bigg(\bigcap\limits_{l=s+1}^n\left\{X_{j_l}>m\right\}\bigg)\right),  \label{th1t1}
\end{eqnarray}
or, equivalently,
\begin{eqnarray}
EX^p_{r:n}&= &\sum_{m=0}^\infty \big((m+1)^p-m^p\big)
\left\{1-\sum\limits_{s=r}^{n}\sum_{(j_1,\ldots,j_n)\in\mathcal{P}_s} \right.\nonumber \\
&& \left.P\left(\bigg(\bigcap\limits_{l=1}^s\left\{X_{j_l}\leq m\right\}\bigg)\cap \bigg(\bigcap\limits_{l=s+1}^n\left\{X_{j_l}>m\right\}\bigg)\right)\right\},  \label{th1t2}
\end{eqnarray}
where $\mathcal{P}_s$ denotes the subset of permutations $(j_1,j_2,\ldots,j_n)$ of $(1,2,\ldots,n)$ satisfying
$$j_1<j_2<\cdots<j_n \hbox{ and } j_{s+1}<j_{s+2}<\cdots<j_n,$$
and it is understood that $\mathcal{P}_0=\mathcal{P}_n=\{(1,2,\ldots,n)\}$.
\end{theorem}
\begin{proof}
We use the following well known fact:  if a discrete RV $X$ has a~support consisting only  of non-negative integers, then
\begin{equation}
\label{fact1_formula}
EX^p=p\int\limits_0^\infty x^{p-1}P(X>x)dx=\sum_{m=0}^\infty \big((m+1)^p-m^p\big)P(X>m), \quad p=1,2,\ldots.
\end{equation}
This fact implies that  if $X_1,X_2,\ldots,X_n$ are DNID RV's taking values in the set of non-negative integers, then, for $1\leq r\leq n$ and $p=1,2,\ldots$,
\begin{equation}
\label{th1d1}
EX_{r:n}^p=\sum_{m=0}^\infty \big((m+1)^p-m^p\big)P\left(X_{r:n}>m\right).
\end{equation}
Thus we are reduced to finding $P\left(X_{r:n}>m\right)$. For this purpose observe that
$$\{X_{r:n}>m\}=\bigcup_{s=0}^{r-1}A_s \; \hbox{ and }  \; \{X_{r:n}\leq m\}=\bigcup_{s=r}^{n}A_s,$$
where the events  $A_s$, $s=0,1,\ldots,n$, given by
$$A_s=\left\{\textrm{exactly }s\textrm{ of }X_i\textrm{ are }\leq m\textrm{ and the remaining }n-s\textrm{ of }X_i\textrm{ are }>m \right\}$$
are pairwise disjoint. Therefore
\begin{equation}
\label{th1d2}
P(X_{r:n}>m)=\sum_{s=0}^{r-1}P\left(A_s\right),
\end{equation}
or equivalently,
\begin{equation}
\label{th1d3}
P(X_{r:n}>m)=1-P(X_{r:n}\leq m)=1-\sum_{s=r}^{n}P\left(A_s\right).
\end{equation}
But, for $s=0,1,\ldots,n$,
\begin{equation}
\label{th1d4}
P\left(A_s\right)=\sum_{(j_1,\ldots,j_n)\in\mathcal{P}_s}P\left(\bigg(\bigcap\limits_{l=1}^s\left\{X_{j_l}\leq m\right\}\bigg)\cap \bigg(\bigcap\limits_{l=s+1}^n\left\{X_{j_l}>m\right\}\bigg)\right).
\end{equation}
Combining (\ref{th1d1}) with  (\ref{th1d2})-(\ref{th1d4}) yields  (\ref{th1t1}) and  (\ref{th1t2}). The proof is complete.
\end{proof}

\begin{remark}
\label{rem1}

{\bf (1)} Under the additional assumption that $X_1,X_2,\ldots,X_n$  are independent and $X_i$ has cumulative distribution function $F_i(x)=P(X_i\leq x)$ and survival function $\bar{F}_i(x)=P(X_i> x)$, $i=1,\ldots,n$, formulas (\ref{th1t1}) and  (\ref{th1t2}) reduce to 
$$EX^p_{r:n}=\sum\limits_{m=0}^\infty \big((m+1)^p-m^p\big)\sum_{s=0}^{r-1}
\sum_{(j_1,\ldots,j_n)\in\mathcal{P}_s} \bigg(\prod_{l=1}^sF_{j_l}(m)\bigg) \bigg(\prod_{l=s+1}^n\bar{F}_{j_l}(m)\bigg)
$$
and
$$EX^p_{r:n}= \sum_{m=0}^\infty \big((m+1)^p-m^p\big)
\left\{1-\sum\limits_{s=r}^{n}\sum_{(j_1,\ldots,j_n)\in\mathcal{P}_s}\bigg(\prod_{l=1}^sF_{j_l}(m)\bigg) \bigg(\prod_{l=s+1}^n\bar{F}_{j_l}(m)\bigg)\right\},
$$
respectively. Thus we recover  \citep[Theorem 2.1]{DD18}.

\noindent
{\bf (2)} Let $\mathcal{P}$ denote the set of all permutations $(j_1,j_2,\ldots,j_n)$ of $(1,2,\ldots,n)$. If in Theorem \ref{th1} we additionally require that the random vector $(X_1,X_2,\ldots,X_n)$ is exchangeable, that is, for $(j_1,j_2,\ldots,j_n)\in\mathcal{P}$,
$(X_{j_1},X_{j_2},\ldots,X_{j_n})$ has the same distribution as $(X_1,X_2,\ldots,X_n)$,
then, for all $(j_1,j_2,\ldots,j_n)\in\mathcal{P}$,
\begin{eqnarray*}
P\left(\bigg(\bigcap\limits_{l=1}^s\left\{X_{j_l}\leq m\right\}\bigg)\cap \bigg(\bigcap\limits_{l=s+1}^n\left\{X_{j_l}>m\right\}\bigg)\right) && \\
=P\left(\bigg(\bigcap\limits_{l=1}^s\left\{X_{l}\leq m\right\}\bigg)\cap \bigg(\bigcap\limits_{l=s+1}^n\left\{X_{l}>m\right\}\bigg)\right). &&
\end{eqnarray*}
Since there are exactly ${n \choose s}$ elements of $\mathcal{P}_s$,  (\ref{th1t1}) and  (\ref{th1t2}) simplify to 
\begin{eqnarray*}
 EX^p_{r:n} &=& \sum\limits_{m=0}^\infty \big((m+1)^p-m^p\big)\\
&& \times\sum_{s=0}^{r-1} {n\choose s}
P\left(\bigg(\bigcap\limits_{l=1}^s\left\{X_{l}\leq m\right\}\bigg)\cap \bigg(\bigcap\limits_{l=s+1}^n\left\{X_{l}>m\right\}\bigg)\right), 
\end{eqnarray*}
and

\begin{eqnarray*}
EX^p_{r:n}&=& \sum_{m=0}^\infty \big((m+1)^p-m^p\big)  \\
&& \times \left\{1-\sum\limits_{s=r}^{n}
{n\choose s}
P\left(\bigg(\bigcap\limits_{l=1}^s\left\{X_{l}\leq m\right\}\bigg)\cap \bigg(\bigcap\limits_{l=s+1}^n\left\{X_{l}>m\right\}\bigg)\right)\right\},
\end{eqnarray*}
respectively.

\noindent
{\bf (3)}
If $r<\frac{n+1}{2}$ then it is better to use (\ref{th1t1}) rather than  (\ref{th1t2}), because it requires a smaller number of arithmetic operation. If in turn $r>\frac{n+1}{2}$ then  (\ref{th1t2}) leads to shorter time of computation than  (\ref{th1t1}). 
 In particular, applying  (\ref{th1t1}) with $r=1$ and  (\ref{th1t2}) with $r=n$ gives, for $p=1,2,\ldots$,
\begin{equation}
\label{firstOS}
EX^p_{1:n}= \sum\limits_{m=0}^\infty \big((m+1)^p-m^p\big) P\left(\bigcap\limits_{l=1}^n\{X_l>m\}\right),
\end{equation}
\begin{equation}
\label{largestOS}
EX^p_{n:n}=
\sum\limits_{m=0}^\infty \big((m+1)^p-m^p\big)
\left\{1-P\left(\bigcap\limits_{l=1}^n\{X_l\leq m\}\right)\right\}.
\end{equation}

\noindent
{\bf (4)} In \citep[Section 3.2]{E17} a formula for the cumulative distribution function of $X_{r:n}$ has been derived under the assumption that $(X_1,X_2,\ldots,X_n)$ consists of multiple types of RV's such that the RV's of the same type are exchangeable and the RV's of different types are arbitrary dependent. This formula can be alternatively used to obtain moments of $X_{r:n}$ whenever the above assumption is fulfilled.
\end{remark}

\medskip
Theorem \ref{th1} is very general in the sense that it provides formulas for moments of order statistics from discrete RV's $X_1,X_2,\ldots,X_n$ under the single assumption that $X_i$'s take values in the set of non-negative integers. Theoretically we can use this theorem for any marginal distributions of $X_i$'s (with supports containing only non-negative integers) and for any dependence structure between $X_i$'s. Yet, in practice formulas (\ref{th1t1}) and  (\ref{th1t2}) work well only when the supports of $X_1,X_2,\ldots,X_n$ are finite - then the infinite sums $\sum_{m=0}^\infty$ in (\ref{th1t1}) and  (\ref{th1t2}) consist of finite numbers of non-zero elements which allows their evaluation using software. To illustrate with a numerical example application of Theorem \ref{th1} in the case of finite supports of $X_1,X_2,\ldots,X_n$, we consider the random vector $(X_1,\ldots,X_{10})$ with multinomial distribution $Mult(20,(p_1,\ldots,p_{10}))$, i.e., we assume that
$$
P(X_i=n_i, i=1,\ldots,10)=\left\{
\begin{array}{cl}
\frac{20!}{n_1!\ldots n_{10}!}p_1^{n_1}\ldots p_{10}^{n_{10}} & \hbox{ if } \; \sum_{i=1}^{10}n_i=20 \\
0 & \hbox{ otherwise}
\end{array},
\right.
$$
where $p_i>0$, $i=1,\ldots,10$, and $\sum_{i=1}^{10}p_i=1$. In Table \ref{Table1}, for various values of $p_i>0$, $i=1,\ldots,10$, we present means, second raw moments (in brackets) and variances (in double brackets) of the corresponding order statistics $X_{r:10}$, $1\leq r\leq 10$.

\bigskip
\begin{landscape}
\begin{table}[ht]
\footnotesize
\caption{Mean, second raw moment (in brackets) and variance (in double brackets) of $X_{r:10}$ from $(X_1,X_2,\ldots,X_{10})\sim Mult(20,(p_1,p_2,\ldots,p_{10}))$}

\begin{tabular}{|l||c|c|c|c|c|c|c|c|c|c||}
\hline
$p_i\backslash r$ &1 & 2& 3 & 4 &5 &6 &7 &8 &9 &10\\ \hline\hline
\multirow{3}{*}{$p_1=p_2=\cdots=p_{10}=0.1$} &0.215&0.654&0.991&1.325&1.733&2.011&2.368&2.873&3.421&4.410\\
                                            &(0.215)&(0.662)&(1.120)&(1.987)&(3.203)&(4.148)&(5.847)&(8.477)&(12.048)&(20.292)\\
                                            &((0.169))&((0.234))&((0.139))&((0.233))&((0.199))&((0.104))&((0.240))&((0.226)) &((0.343))&((0.846))\\\hline

$p_1=\cdots=p_5=0.08, $ &0.182&0.600&0.953&1.280&1.691&1.998&2.373&2.892&3.484&4.547\\
\multirow{2}{*}{$p_6=\cdots=p_{10}=0.12$}                      &(0.182)&(0.606)&(1.057)&(1.862)&(3.077)&(4.112)&(5.877)&(8.604)&(12.519)&(21.625)\\
                                                            &((0.149))&((0.245))&((0.149))&((0.224))&((0.218))&((0.120))&((0.246))&((0.240))&((0.379))&((0.950))\\\hline

$p_i=0.045+0.01i$, &0.148 &0.540&0.911&1.236&1.648&  1.985&2.381&2.916&3.551&4.683\\
\multirow{2}{*}{$i=1,\ldots,10$} &(0.148)&(0.544)&(0.993)&(1.742)&(2.950)&(4.078)&(5.924)&(8.758)&(13.024)&(22.973)\\
                                                    &((0.126))&((0.252))&((0.164))&((0.215))&((0.234))&((0.138))&((0.253))&((0.254))&((0.413))&((1.043))\\\hline

$p_1=\cdots=p_9=0.05,$  &0.009&0.077&0.284&0.597&0.877&1.118&1.451&1.914&2.672&11.001\\
\multirow{2}{*}{$p_{10}=0.55$} &(0.009)&(0.077)&(0.284)&(0.603)&(0.934)&(1.423)&(2.385)&(4.011)&(7.815)&(125.959)\\
                                                    &((0.008))&((0.071))&((0.204))&((0.246))&((0.165))&((0.174))&((0.280))&((0.348))&((0.674))&((4.938))\\\hline
\end{tabular}
\label{Table1}
\end{table}
\end{landscape}
A problem arises when we want to apply Theorem \ref{th1} to compute moments of order statistics from RV's with infinite supports. Then the sums $\sum_{m=0}^\infty$ in (\ref{th1t1}) and  (\ref{th1t2}) have infinitely many positive terms and we are not able to add all these terms using software. We propose two solutions of this problem.  The first solution  is a truncation method presented below and leading to approximate values of $EX^p_{r:n}$. The second one,  described in details in Section~\ref{section3}, concerns a special  case when the random vector $(X_1,X_2,\ldots,X_n)$ has a multivariate geometric distribution.

In the truncation method we first fix $d>0$, the desired accuracy of the result. Next we split the infinite series in (\ref{th1d1}) into two parts as follows
\begin{eqnarray*}
EX_{r:n}^p&=&\sum_{m=0}^{M_0} \big((m+1)^p-m^p\big)P\left(X_{r:n}>m\right) \\
&+&\sum_{m=M_0+1}^\infty \big((m+1)^p-m^p\big)P\left(X_{r:n}>m\right),
\end{eqnarray*}
where $M_0$ is so chosen that
\begin{equation}
\label{s2M0}
\sum_{m=M_0+1}^\infty \big((m+1)^p-m^p\big)P\left(X_{r:n}>m\right)\leq d.
\end{equation}
Then by (\ref{th1d2}) - (\ref{th1d4}) both  the equivalent  approximate formulas
\begin{eqnarray}
EX^p_{r:n}&\approx&\sum\limits_{m=0}^{M_0} \big((m+1)^p-m^p\big)  \sum_{s=0}^{r-1}  \sum_{(j_1,\ldots,j_n)\in\mathcal{P}_s}   \nonumber \\
&&
P\left(\bigg(\bigcap\limits_{l=1}^s\left\{X_{j_l}\leq m\right\}\bigg)\cap \bigg(\bigcap\limits_{l=s+1}^n\left\{X_{j_l}>m\right\}\bigg)\right)  \label{af1}
\end{eqnarray}
and
\begin{eqnarray}
EX^p_{r:n}&\approx&\sum\limits_{m=0}^{M_0} \big((m+1)^p-m^p\big)
\left\{1-\sum\limits_{s=r}^{n}\sum_{(j_1,\ldots,j_n)\in\mathcal{P}_s} \right.\nonumber \\
&& \left.P\left(\bigg(\bigcap\limits_{l=1}^s\left\{X_{j_l}\leq m\right\}\bigg)\cap \bigg(\bigcap\limits_{l=s+1}^n\left\{X_{j_l}>m\right\}\bigg)\right)\right\}  \label{af2}
\end{eqnarray}
introduce an error not greater than $d$.

What is left is to find $M_0$  satisfying (\ref{s2M0}). First observe that if  $EX^p_{r:n}$ is finite then the infinite series in (\ref{th1d1}) converges and consequently for any $d>0$ there exists  $M_0$ such that (\ref{s2M0}) holds. This means that the finiteness of $EX^p_{r:n}$ guarantees the existence of $M_0$  satisfying (\ref{s2M0}). Yet, to derive a~convenient formula for  $M_0$ we will need a stronger assumption than $EX^p_{r:n}<\infty$, namely condition (\ref{th2a1}) given later on. To see this note that 
\begin{eqnarray}
&&\sum_{m=M_0+1}^\infty \big((m+1)^p-m^p\big)P\left(X_{r:n}>m\right)  \nonumber\\
&=&\sum_{m=M_0+1}^\infty \big((m+1)^p-m^p\big) \nonumber \\
&& \times\sum_{s=0}^{r-1}
\sum_{(j_1,\ldots,j_n)\in\mathcal{P}_s}P\left(X_{j_1}\leq m, \ldots, X_{j_s}\leq m, X_{j_{s+1}}> m,\ldots,X_{j_n}> m\right)   \nonumber\\
&\leq& \sum_{m=M_0+1}^\infty \big((m+1)^p-m^p\big)\sum_{s=0}^{r-1}
\sum_{(j_1,\ldots,j_n)\in\mathcal{P}_s}P\left(X_{j_n}> m\right)    \nonumber\\
&\leq& \sum_{m=M_0+1}^\infty \big((m+1)^p-m^p\big) \max_{j=1,\ldots,n}P\left(X_j> m\right) 
\sum_{s=0}^{r-1}
\sum_{(j_1,\ldots,j_n)\in\mathcal{P}_s}1   \nonumber\\
&=& \sum_{m=M_0+1}^\infty \big((m+1)^p-m^p\big) \max_{j=1,\ldots,n}P\left(X_j> m\right) 
\sum_{s=0}^{r-1}
{n \choose s}   \nonumber\\
&=&\sum_{s=0}^{r-1}{n \choose s}  \sum_{m=M_0+1}^\infty \big((m+1)^p-m^p\big) \sum_{x=m+1}^{\infty}P(X_{j_{max(m)}}=x), 
\label{s2ps1}
\end{eqnarray}
where 
\begin{equation}
\label{jmaxm}
j_{max(m)}=\argmax_{j=1,\ldots,n}P\left(X_j> m\right).
\end{equation}
In the exchangeable case all the probabilities $P\left(X_j> m\right)$, $j=1,\ldots,n$, are equal so we can take $j_{max(m)}=1$, $m=0,1,\ldots$. Hence in this case $j_{max(m)}$ does not depend on $m$. This is so also in some non-exchangeable cases of interest (see the begining of proofs of Corollaries \ref{th2c1} and \ref{th2c2}). Assuming that $j_{max(m)}$ does not depend on $m$, denoting 
$$j_{max(m)}=j_0, \quad m=0,1,\ldots$$ 
and interchanging the order of summation in (\ref{s2ps1}) we get
\begin{eqnarray*}
&&\sum_{m=M_0+1}^\infty \big((m+1)^p-m^p\big)P\left(X_{r:n}>m\right) \\
&\leq&\sum_{s=0}^{r-1}{n \choose s} 
\sum_{x=M_0+2}^{\infty}P(X_{j_0}=x)
 \sum_{m=M_0+1}^{x-1} \big((m+1)^p-m^p\big) \\ 
&=&\sum_{s=0}^{r-1}{n \choose s} 
\sum_{x=M_0+2}^{\infty}P(X_{j_0}=x)
\big(x^p-(M_0+1)^p\big) \\ 
&\leq&\sum_{s=0}^{r-1}{n \choose s} 
\sum_{x=M_0+2}^{\infty}x^pP(X_{j_0}=x).
\end{eqnarray*}
It follows that the condition 
\begin{equation}
\label{warM0}
 \sum_{x=M_0+2}^{\infty}x^pP(X_{j_0}=x)\leq d \left(\sum_{s=0}^{r-1}{n \choose s}\right)^{-1}
\end{equation}
implies (\ref{s2M0}). Note that for any $d>0$ there exists  $M_0$ satisfying (\ref{warM0}) if $EX^p_{j_0}<\infty$, because then $\lim_{M_0\to\infty} \sum_{x=M_0+2}^{\infty}x^pP(X_{j_0}=x)=0$.

Summarizing, we have proved the following result.
\begin{theorem}
\label{th2}
Let the assumptions of Theorem \ref{th1} hold. Moreover, suppose that $j_{max(m)}$ defined in (\ref{jmaxm}) does not depend on $m$ and denote it briefly by $j_0$. If, for fixed $p\in\{1,2,\ldots\}$,
\begin{equation}
\label{th2a1}
EX^p_{j_0}<\infty,
\end{equation}
then, for  $1\leq r\leq n$, $EX^p_{r:n}$ is finite and the approximate formulas (\ref{af1}) and (\ref{af2}) with $M_0$ satisfying (\ref{warM0}) introduce an error not greater than $d$.
\end{theorem}

Condition (\ref{warM0}) determining $M_0$ can be rewritten in an explicit form  for some specific univariate marginal distributions of $(X_1,\ldots,X_n)$. Below we consider two cases: the first one when these marginal distributions are Poisson, and the second one when these are negative binomial. By $F_X^{\leftarrow}$ we denote the quantile function of the RV $X$, i.e., 
$$
F_X^{\leftarrow}(q)=\min\{x\in\mathbb{R} : P(X\leq x)\geq q\}, \quad q\in(0,1).
$$

\begin{corollary}
\label{th2c1}
Let the marginal distributions of the random vector \linebreak  $(X_1,\ldots,X_n)$ be Poisson distributions $\mathrm{Pois}(\lambda_j)$ with $\lambda_j>0$, $j=1,\ldots,n$,  i.e., 
$$P(X_j=x)=\e^{-\lambda_j}\frac{\lambda_j^x}{x!}, \quad x=0,1,\ldots, \;  j=1,\ldots, n.$$
Set $ j_0=\argmax_{j=1,\ldots,n}\lambda_j$. 
Then the approximate formulas (\ref{af1}) and (\ref{af2})  with 
$$
M_0=\left\{
\begin{array}{lcl}
p-2 & \hbox{if} &  d_{Pois}\leq 0 \\
F_{X_{j_0}}^{\leftarrow}(d_{Pois})+p-1  & \hbox{if} &  d_{Pois}\in(0,1)
\end{array}
\right.,
$$
where 
$$
d_{Pois}=1-d\, 2^{-p(p-1)/2}\lambda_{j_0}^{-p}\left(\sum_{s=0}^{r-1}{n \choose s}\right)^{-1},
$$
introduce an error not exceeding the fixed value $d>0$.
\end{corollary}
\begin{proof}
If $X_j\sim \mathrm{Pois}(\lambda_j)$, then
$$
P(X_j>m)=1+\left(-\e^{-\lambda_j}\right) \sum_{x=0}^m\frac{\lambda_j^x}{x!}, \quad m=0,1,\ldots,
$$
which shows that, for any fixed $m\in\{0,1,\ldots\}$, $P(X_j>m)$ is an increasing function of $\lambda_j>0$. Consequently, for $m=0,1,\ldots$,
$$
\max_{j=1,\ldots,n}P(X_j>m)=P(X_{j_0}>m), \quad \hbox{ where } j_0=\argmax_{j=1,\ldots,n}\lambda_j.
$$
Hence $j_{max(m)}$ defined in (\ref{jmaxm}) does not depend on $m$ and we have $j_{max(m)}=j_0$. Therefore the assumptions of Theorem \ref{th2} are satisfied. Before we apply this theorem we show that if $M_0\geq p-2$, then the condition
\begin{equation}
\label{M0Pois}
P(X_{j_0}>M_0+1-p)\leq d\, 2^{-p(p-1)/2}\lambda_{j_0}^{-p}\left(\sum_{s=0}^{r-1}{n \choose s}\right)^{-1},
\end{equation}
 implies (\ref{warM0}). To do this, note that, for any $y=2,3,\ldots$ and $q=0,1,\ldots$, 
\begin{equation}
\label{nierpom}
y^q=(y-1+1)^q=\sum_{i=0}^q{q \choose i}(y-1)^i\leq  \sum_{i=0}^q{q \choose i}(y-1)^q=(y-1)^q2^q.
\end{equation}
Consequently,
$$
\begin{array}{lcll}
&& \hspace{-7mm}\displaystyle \sum_{x=M_0+2}^\infty x^pP(X_{j_0}=x) = \sum_{x=M_0+2}^\infty\e^{-\lambda_{j_0}} \lambda_{j_0}^x\frac{x^{p-1}}{(x-1)!}&\\
&\leq & \displaystyle 2^{p-1}\sum_{x=M_0+2}^\infty\e^{-\lambda_{j_0}} \lambda_{j_0}^x\frac{(x-1)^{p-2}}{(x-2)!}  & \hbox{ if } p\geq 2\\ 
&\leq & \displaystyle  2^{p-1}2^{p-2}\sum_{x=M_0+2}^\infty\e^{-\lambda_{j_0}} \lambda_{j_0}^x\frac{(x-2)^{p-3}}{(x-3)!}  & \hbox{ if } p\geq 3 \hbox{ and } M_0\geq 1\\ 
&\leq & \displaystyle \cdots &\\
&\leq & \displaystyle 2^{p\!-\!1}2^{p\!-\!2}\cdots 2^{p\!-\!(p\!-\!1)}\sum_{x=M_0+2}^\infty\e^{-\lambda_{j_0}} \lambda_{j_0}^x\frac{(x\!-\!(p\!-\!1))^{p-p}}{(x-p)!}   & \hbox{ if }  M_0\geq p-2\\ 
&=& \displaystyle 2^{p(p-1)/2} \lambda_{j_0}^p \sum_{x=M_0+2}^\infty\e^{-\lambda_{j_0}} \frac{\lambda_{j_0}^{x-p}}{(x-p)!} &\\
&=& \displaystyle 2^{p(p-1)/2} \lambda_{j_0}^pP(X_{j_0}>M_0+1-p).&
\end{array}
$$
Thus
 indeed if $M_0\geq p-2$ and (\ref{M0Pois})  holds then (\ref{warM0}) is satisfied. But if $d_{Pois}\in(0,1)$ then the smallest $M_0$ for which (\ref{M0Pois})  is true equals $F_{X_{j_0}}^{\leftarrow}(d_{Pois})+p-1$. Now 
application of Theorem \ref{th2} finishes the proof.
\end{proof}

\begin{corollary}
\label{th2c2}
Let the marginal distributions of the random vector \linebreak $(X_1,\ldots,X_n)$ be negative binomial  distributions $\mathrm{NBin}(R,p_j)$ with $R>0$ and $p_j\in(0,1)$, $j=1,\ldots,n$, i.e.,
$$P(X_j=x)=\frac{\Gamma(x+R)}{x!\Gamma(R)}(1-p_j)^xp_j^R, \quad x=0,1,\ldots, \;  j=1,\ldots, n.$$
Set $j_0=\argmin_{j=1,\ldots,n}p_j$ and 
\begin{equation}
\label{defXfalka}
\tilde{X}\sim \mathrm{NBin}(R+p,p_{j_0}).
\end{equation}
Then the approximate formulas (\ref{af1}) and (\ref{af2})  with 
$$
M_0=\left\{
\begin{array}{lcl}
p-2 & \hbox{if} &  d_{NBin}\leq 0 \\
F_{\tilde{X}}^{\leftarrow}(d_{NBin})+p-1  & \hbox{if} &  d_{NBin}\in(0,1)
\end{array}
\right.,
$$
where 
$$
d_{NBin}=1-\frac{d}{2^{p(p-1)/2}\prod_{i=0}^{p-1}(R+i)\sum_{s=0}^{r-1}{n \choose s}}\left(\frac{p_{j_0}}{1-p_{j_0}}\right)^p,
$$
 introduce an error not exceeding the fixed value $d>0$.
\end{corollary}
\begin{proof}
The proof is similar to that of Corollary \ref{th2c1}. First we observe that, for any fixed $R>0$ and $m\in\{0,1,\ldots\}$, $P(X_j>m)$ is a decreasing function of $p_j\in(0,1)$, because
$$
P(X_j>m)=1+\left(-p_j^R\right) \sum_{x=0}^m\frac{\Gamma(x+R)}{x!\Gamma(R)}(1-p_j)^x, \quad m=0,1,\ldots.
$$
 Consequently, for $m=0,1,\ldots$,
$$
\max_{j=1,\ldots,n}P(X_j>m)=P(X_{j_0}>m), \quad \hbox{ where } j_0=\argmin_{j=1,\ldots,n}p_j,
$$
which shows that $j_{max(m)}$ defined in (\ref{jmaxm}) does not depend on $m$ and  $j_{max(m)}=j_0$. 

Now using (\ref{nierpom}) we get
$$
\begin{array}{lcll}
&&\hspace{-7mm}  \displaystyle \sum_{x=M_0+2}^\infty x^pP(X_{j_0}=x)=\sum_{x=M_0+2}^\infty x^{p-1}\frac{\Gamma(x+R)}{(x-1)!\Gamma(R)}(1-p_{j_0})^xp_{j_0}^R & \\
&\leq & \displaystyle 2^{p-1}\sum_{x=M_0+2}^\infty (x-1)^{p-2}\frac{\Gamma(x+R)}{(x-2)!\Gamma(R)}(1-p_{j_0})^xp_{j_0}^R  & \hbox{ if }   p\geq 2\\ 
&\leq & \displaystyle \cdots &\\
&\leq &  \displaystyle 2^{p(p-1)/2} \sum_{x=M_0+2}^\infty
 \frac{\Gamma(x+R)}{(x-p)!\Gamma(R)}(1-p_{j_0})^xp_{j_0}^R  
& \hbox{ if } M_0\geq p-2 \\
&=& \displaystyle 2^{p(p-1)/2}R(R+1)\cdots(R+p-1) \left(\frac{1-p_{j_0}}{p_{j_0}}\right)^p & \\
&& \qquad \times P(\tilde{X}>M_0+1-p),&
\end{array}
$$
where the RV $\tilde{X}$ is defined in (\ref{defXfalka}). It follows that  if $M_0\geq p-2$ then the condition 
\begin{equation*}
P(\tilde{X}>M_0+1-p)\leq  \frac{d}{2^{p(p-1)/2}R(R+1)\cdots(R+p-1)\sum_{s=0}^{r-1}{n \choose s}}\left(\frac{p_{j_0}}{1-p_{j_0}}\right)^p,
\end{equation*}
 implies (\ref{warM0}). Theorem \ref{th2} now yields the desired conclusion.
\end{proof}

The quantile functions of RV's with Poisson and negative binomial distributions are implemented in statistical packages. Therefore Corollaries \ref{th2c1} and \ref{th2c2} enable to compute numerically approximate values of raw moments of order statistics from the random vector $(X_1,X_2,\ldots,X_n)$ with univariate marginal Poisson or negative binomial distributions. Moreover,  Corollaries \ref{th2c1} and \ref{th2c2} can be applied in the case of any dependence structure between $X_1,\ldots,X_n$. To illustrate the computational details  we fix $n=10$, the level of accuracy $d=0.0005$, and assume that $X_1,\ldots,X_{10}$ are independent. Selecting some values of $\lambda_i$,  $i=1,\ldots,10$, in Tables \ref{Table2} and \ref{Table3}, we present means $EX_{r:10}$ and second raw moments $EX^2_{r:10}$, respectively, of all order statistics from  $(X_1,\ldots,X_{10})$, where $X_i\sim \mathrm{ Pois}(\lambda_i)$, $i=1,\ldots,10$. Furthermore, in each table  in brackets underneath  moments, we provide values of $M_0$, i.e., numbers of terms sufficient in the sum to obtain the desired accuracy. Corresponding results for $(X_1,\ldots,X_{10})$ with $X_i$,  $i=1,\ldots,10$, having the negative binomial distribution $\mathrm{NBin}(R,p_{i})$ are given in Tables \ref{Table4} and \ref{Table5} for $R=2$, $R=5$ and selected values of $p_i$, $i=1,\ldots,10$.

\bigskip
\begin{landscape}
\begin{table}[ht]
\footnotesize
\begin{center}
\caption{Mean of $X_{r:10}$ from $(X_1,\ldots,X_{10})$, where $X_i\sim$Pois($\lambda_i$) and $X_{1},\ldots,X_{10}$ are independent}
\begin{tabular}{|l||c|c|c|c|c|c|c|c|c|c||}
\hline
$\lambda_i\backslash r$ &1 & 2& 3 & 4 & 5 & 6 & 7 & 8 & 9 & 10\\ \hline\hline
$\lambda_i=1, i=1,\ldots,10$ & 0.010 & 0.070 & 0.225& 0.471& 0.737& 0.979& 1.230 & 1.551& 1.990& 2.738\\
                                & (6) & (7)  & (8) & (8)    & (8) & (9) & (9) & (9) & (9) & (9)\\ \hline
$\lambda_i=1, i=1,\ldots,5$ & 0.081  & 0.343& 0.722&1.117 &1.557 & 2.116&2.864 &3.851 &5.155 &7.193 \\
$\lambda_i=i-4, i=6,\ldots,10$& (17) & (19) & (20) & (21) & (22) & (22) & (23) & (23) &(23) & (23)\\ \hline
$\lambda_i=1, i=1,\ldots,5$ & 0.102 &0.414 &0.860 &1.389 &2.220 &6.497 & 8.367 &9.879 &11.483 &13.788 \\
$\lambda_i=10, i=6,\ldots,10$& (24) & (27) &(28) &(29)   &(30) &(31) & (31)     &(31) & (31) & (31)\\ \hline
$\lambda_i=i, i=1,\ldots,10$ &0.620 & 1.598&2.587 &3.585 &4.601 &5.653 & 6.774 &8.030 &9.578 & 11.974\\
                             &(24) &(27)   &(28) &(29)    & (30) &(31) & (31) & (31)  &(31) & (31)\\ \hline
$\lambda_i=2i+1, i=1,\ldots,4$ &2.482 &4.354 &5.806 &6.969 &7.980 &8.934 &9.901 &10.963 &12.272 & 14.339\\
$\lambda_i=10, i=5,\ldots,10$& (24) &(27)  &(28) &(29)    &(30)  &(31)  &(31)  &(31) &(31) & (31)\\ \hline
$\lambda_1=\lambda_2=\lambda_3=10$& 7.375&9.844 &12.339 &16.587 &19.696 &22.727 &26.539  &30.155 &34.638 &50.099 \\
$\lambda_4=\lambda_5=\lambda_6=20$& (83)&(87)   &(90)   &(92) &(93)      &(94)  &(94)  & (94)    & (94) & (94)\\
$\lambda_7=\lambda_8=\lambda_9=30$& & & & & & &  & & & \\
$\lambda_{10}=50$& & & & & & &  & & & \\ \hline
\end{tabular}
\label{Table2}
\end{center}
\end{table}

\begin{table}[ht]
\footnotesize
\begin{center}
\caption{Second raw moments of $X_{r:10}$ from $(X_1,\ldots,X_{10})$, where $X_i\sim$Pois($\lambda_i$) and $X_{1},\ldots,X_{10}$ are independent}
\begin{tabular}{|l||c|c|c|c|c|c|c|c|c|c||}
\hline
$\lambda_i\backslash r$ &1 & 2& 3 & 4 & 5 & 6 & 7 & 8 & 9 & 10\\ \hline\hline
$\lambda_i=1, i=1,\ldots,10$& 0.010&0.070 &0.227 &0.480 &0.789 &1.173 &1.770 &2.751 &4.412 &8.319 \\
                & (7)  &(8)  & (9)  & (9)   &(10)  &(10) & (10)  &(10)  &(10) & (10)\\ \hline
$\lambda_i=1, i=1,\ldots,5$ & 0.082 &0.360 &0.839 &1.585 &2.848 &5.042 &9.030 &16.084 &28.522 &55.608 \\
$\lambda_i=i-4, i=6,\ldots,10$& (20) &(22) &(23)  &(24) & (25)  &(25) & (25)  &(25)     &(25) & (25)\\\hline
$\lambda_i=1, i=1,\ldots,5$ & 0.105&0.453 &1.116 &2.419 &5.809 &45.464 &72.835 &100.538 &135.397 &195.864 \\
$\lambda_i=10, i=6,\ldots,10$& (28)&(31)  & (32)& (33) & (34)  &(34)   & (34)  & (34)   & (34)& (34)\\\hline
$\lambda_i=i, i=1,\ldots,10$ & 0.870 &3.318 &7.636 &13.961 &22.455 &33.449 &47.660  &66.701 &94.812 & 149.138\\
&(28) &(31) & (32) &(33)  & (34)   & (34)  & (34)   & (34)  & (34)& (34)\\ \hline
$\lambda_i=2i+1, i=1,\ldots,4$ &7.922 &20.733 &35.407 &50.229 &65.376 &81.593 &99.979 &122.459 & 153.556& 210.746\\
$\lambda_i=10,  i=5,\ldots,10$ &(28)  &(31)   & (32)  & (33)  & (34)  &  (34) & (34)  & (34)& (34)  & (34)\\ \hline
$\lambda_1=\lambda_2=\lambda_3=10$&58.889 &101.184 &157.427 &282.417 &395.389 &524.549 &714.111 &921.182 &1217.132 & 2557.719\\
$\lambda_4=\lambda_5=\lambda_6=20$&(92) & (96)     & (98)   & (100)  & (101)  & (102)  & (102)  & (102)  & (102) & (102)\\
$\lambda_7=\lambda_8=\lambda_9=30$& & & & & & &  & & & \\
$\lambda_{10}=50$& & & & & & &  & & & \\ \hline
\end{tabular}
\label{Table3}
\end{center}
\end{table}
\end{landscape}

\bigskip
\begin{landscape}
\begin{table}[ht]
\begin{center}
\footnotesize
\caption{Mean of $X_{r:10}$ from $(X_1,\ldots,X_{10})$, where $X_i\sim$ NBin($R$,$p_i$) and $X_{1},\ldots,X_{10}$ are independent}
\begin{tabular}{|c|l||c|c|c|c|c|c|c|c|c|c||}
\hline
$R$ & $p_i$ &1 & 2& 3 & 4 & 5 & 6 & 7 & 8 & 9 & 10\\ \hline\hline
\multirow{8}{*}{2} &$p_i=0.1i-0.05$ & 0.003&0.049&0.248&0.665&1.268&2.129&3.500&6.017&12.024&39.429\\
                    &$i=1,\ldots,10$& (271)&(321)&(354)&(378)&(394)&(404)&(410)&(413)&(414)&(414)\\\cline{2-12}
                    &$p_i=0.25$ &0.768    &1.708&2.617&3.534&4.512&5.603&6.883&8.497&10.788&15.090\\
                    &$i=1,\ldots,10$& (41)&(50) &(56) &(60) &(63) &(64) &(66) &(66) &(66)  &(66)\\\cline{2-12}
                    &$p_i=0.25, i=1,\ldots,8$, &0.409&1.112&1.888&2.716&3.633&4.685&5.946&7.556&9.859&14.194\\
                    & $p_9=p_{10}=0.5$           &(41) &(50) &(56) &(60) &(63) &(64) &(66) &(66) &(66)&(66)\\\cline{2-12}
                    &$p_i=0.25, i=1,\ldots,8$, &0.121&0.562&1.353&2.279&3.304&4.455&5.797&7.469&9.816&14.179\\
                    & $p_9=p_{10}=0,75$  &(41) &(50) &(56) &(60) &(63) &(64) &(66) &(66)& (66)&(66)\\\hline
\multirow{8}{*}{5} &$p_i=0.1i-0.05$ &0.080&0.519&1.302&2.350&3.791&5.889&9.210&15.240&29.435&95.509\\
                    &$i=1,\ldots,10$&(415)&(471)&(509)&(535)&(553)&(564)&(570)&(573) &(574) &(575)\\\cline{2-12}
                    &$p_i=0.25$     &5.295&7.732&9.639&11.387&13.123&14.953&16.998&19.454&22.774&28.644\\
                    &$i=1,\ldots,10$&(64) &(74) &(81) &(86)  &(89)  &(91)  &(92)  &(93)  &(93)  &(93)\\\cline{2-12}
                    &$p_i=0.25, i=1,\ldots,8$, &2.843&4.983&7.084&9.072&11.031&13.057&15.276&17.892&21.363&27.398\\
                    &$p_9=p_{10}=0.5$     &(64)&(74)  &(81) &(86) &(89)  &(91)  &(92)  &(93)  &(93)&(93)\\\cline{2-12}
                    &$p_i=0.25, i=1,\ldots,8$, &0.852&2.296&5.987&8.589&10.806&12.952&15.228&17.872&21.356&27.396\\
                    &$p_9=p_{10}=0.75$          &(64) &(74) &(81) &(86) &(89)  &(91)  &(92)  &(93)  &(93)&(93)\\\hline
\end{tabular}
\label{Table4}
\end{center}
\end{table}
\end{landscape}

\bigskip
\begin{landscape}
\begin{table}[ht]
\footnotesize
\begin{center}
\caption{Second raw moments of $X_{r:10}$ from $(X_1,\ldots,X_{10})$, where $X_i\sim$ NBin($R$,$p_i$) and $X_{1},\ldots,X_{10}$ are independent}
\begin{tabular}{|c|l||c|c|c|c|c|c|c|c|c|c||}
\hline
$R$ & $p_i$ &1 & 2& 3 & 4 & 5 & 6 & 7 & 8 & 9 & 10\\ \hline\hline
\multirow{8}{*}{2} &$p_i=0.1i-0.05$ &0.003&0.050&0.271&0.874&2.327&5.918&15.425&45.583&189.511&2254.318\\
                    &$i=1,\ldots,10$&(369)&(418)&(451)&(474)&(490)&(500)&(506) &(509) &(510)&(510)\\\cline{2-12}
                    &$p_i=0.25$ &1.407   &4.245&8.571&14.671&23.111&34.924&52.086&78.900&127.351&254.734\\
                    &$i=1,\ldots,10$&(51)&(60) &(66) &(70)  &(73)  &(75)  &(76)  &(77)  &(77)   &(77)\\\cline{2-12}
                    &$p_i=0.25, i=1,\ldots,8$, &0.579&2.068&4.737&8.976&15.375&24.943&39.576&63.396&107.904&228.447\\
                    & $p_9=p_{10}=0.5$         &(51) &(60) &(66) &(70) &(73)  &(75)  &(76)  &(77)  &(77)   &(77)\\\cline{2-12}
                    &$p_i=0.25, i=1,\ldots,8$, &0.134&0.772&2.808&6.782&13.235&23.067&38.084&62.335&107.267&228.182\\
                    &$p_9=p_{10}=0.75$      &(51) &(60) &(66) &(70) &(73)  &(75)  &(76)  &(77)  &(77)   &(77)\\\hline
\multirow{8}{*}{5} &$p_i=0.1i-0.05$ &0.084&0.667&2.456&6.867&16.860&39.644&96.143&264.684&1011.931&10968.740\\
                    &$i=1,\ldots,10$&(541)&(595)&(631)&(656)&(673)&(684)  &(691) &(693)  &(694)   &(695)\\\cline{2-12}
                    &$p_i=0.25$     &33.944&65.736&99.251&136.645&180.123&232.831&300.169&393.140&540.481&867.679\\
                    &$i=1,\ldots,10$&(79)  &(89)  &(96)  &(100)  &(103)  &(105)  &(107)  &(107)  &(107)  &(107)\\\cline{2-12}
                    &$p_i=0.25, i=1,\ldots,8$, &11.095&28.637&55.155&88.494&129.197&179.639&244.761&335.147&478.848&799.026\\
                    &$p_9=p_{10}=0.5$  &(79)  &(89)  &(96)  &(100) &(103)  &(105)  &(107)  &(107)  &(107)  &(107)\\\cline{2-12}
                    &$p_i=0.25, i=1,\ldots,8$, &1.562&7.117&42.120&81.087&125.032&177.352&243.566&334.577&478.621&798.965\\
                    &$p_9=p_{10}=0.75$ &(79) &(89) &(96)  &(100) &(103)  &(105)  &(107)  &(107)  &(107)&(107)\\\hline
\end{tabular}
\label{Table5}
\end{center}
\end{table}
\end{landscape}
\bigskip

\section{Multivariate geometric case}
\label{section3}

In this section we will derive formulas convenient for numerical computation of exact values of moments of order statistics from random vectors with multivariate geometric (MVG) distribution.

Before we give the definition of MVG distribution let us recall that the possibly extended RV 
 $X$ is said to be geometrically distributed with parameter $\pi\in[0,1]$ (denoted by $X\sim ge(\pi)$) if  
$$P(X=k)=\pi(1-\pi)^{k}, \quad k=0,1,2,\ldots,$$
or equivalently if  
$$P(X>k)=(1-\pi)^{k+1}, \quad k=-1,0,1,\ldots.$$
 In particular, if $\pi=1$ then $P(X=0)=1$. If $\pi=0$ then $P(X>k)=1$ for $k=-1,0,1,\ldots$, which means that $X$ is an extended RV with defective distribution and $P(X=\infty)=1$. 

\begin{definition}
\label{MVG}
The random vector $(X_1,X_2,\ldots,X_n)$ has the MVG distribution with parameters $\theta_I\in[0,1]$, $I\in J$,  if
\begin{equation*}
X_i=\min\{M_I, I\ni i\},\quad i=1,2,\ldots,n,
\end{equation*}
where
\begin{itemize}
\item[(a)] $J$ is the class of all nonempty subsets of the set $\{1,2,\ldots,n\}$;
\item[(b)] the RV's $M_I$, $I\in J$, are independent;
\item[(c)] for $I\in J$, the RV $M_I$ is geometrically distributed with parameter $1-\theta_I$;
\item[(d)] for every $i=1,2,\ldots,n$, there exists $I$ such that $i\in I$ and $\theta_I\in[0,1)$.
\end{itemize}
\end{definition}

The above definition was introduced by  Esary and Marshall \cite{EM73} to describe random lifetimes of $n$ unrepairable units. These units, denoted by $U_1, U_2,\ldots, U_n$, are exposed to various  shocks that happen in discrete times (cycles) and may cause units failures. More precisely, during each cycle
\begin{itemize}
\item if the  unit  $U_i$ is still operating, it is exposed to a shock which it survives  with probability $\theta_{\{i\}}$ and does not survive with  probability $1-\theta_{\{i\}}$, $1\leq i\leq n$; 
\item  if any of the  units  $U_{i_1}, U_{i_2}$ is still operating, then these two units are  exposed to a~common shock which all the working units among $\{U_{i_1}, U_{i_2}\}$  survive  with probability $\theta_{\{i_1,i_2\}}$ and do not survive with  probability $1-\theta_{\{i_1,i_2\}}$, $1\leq i_1<i_2\leq n$;
\item   if any of the  units  $U_{i_1}, U_{i_2}, U_{i_3}$ is still operating, then these three units are  exposed to a~common shock which all the working units among $\{U_{i_1}, U_{i_2}, U_{i_3}\}$  survive  with probability $\theta_{\{i_1,i_2,i_3\}}$ and do not survive with  probability $1-\theta_{\{i_1,i_2,i_3\}}$, $1\leq i_1<i_2<i_3\leq n$;
$$\vdots$$
\item  if any of the  $n$ units  $U_1, U_2,\ldots, U_n$ is still operating, then all the units are  exposed to a~common shock which the working units   survive  with probability $\theta_{\{1,2,\ldots,n\}}$ and do not survive with  probability $1-\theta_{\{1,2,\ldots,n\}}$.
\end{itemize}
Moreover, it is assumed that before the first cycle all the units are in the working state and that different shocks affects operation of units independently. If $X_i$ denote the number of cycles which the unit $U_i$ survived, $1\leq i\leq n$, then the random vector $(X_1,X_2,\ldots,X_n)$ has the MVG distribution with parameters $\theta_I$, $I\in J$.

Note that Condition (d) of Definition \ref{MVG} ensures that each of the $n$ units will finally break down with probability one. In other words, each RV among $X_1,X_2,\ldots,X_n$ has a~non-defective distribution.

 Esary and Marshall \cite{EM73} gave some properties of the MVG distribution while in \cite{DD19} and \cite{D18} $k$-out-of-$n$ systems with components having some special MVG lifetimes were studied. 
Below we present further properties of the MVG distribution. These will be needed later on in this section to establish  closed-form formulas for factorial moments of order statistics from random vectors with  the MVG distribution.  Furthermore,  these will prove useful in Section \ref{section4} to examine times to failures of coherent systems consisting of elements with MVG lifetimes.
Here, and subsequently, $E\left(Y\right)_p$ denotes the $p$th factorial moment of a RV $Y$, i.e., for $p=1,2,\ldots$,
$$E\left(Y\right)_p=E\big(Y(Y-1)\cdots(Y-p+1)\big)$$
and, for $S=\{l_1,\ldots,l_m\}\subset\{1,\ldots,n\}$, $1\leq m\leq n$, 
$$X_{1:S}=\min\{X_{l_1},\ldots,X_{l_m}\}.$$

\begin{theorem}\label{MVG_properties} 
Let the random vector $(X_1,X_2,\ldots,X_n)$ have the MVG distribution with parameters $\theta_I\in [0,1]$, $\emptyset\neq I\subset \{1,2,\ldots,n\}$.
\begin{itemize}
\item[(i)] For $k_1,k_2,\ldots,k_n\in\{-1,0,1,\ldots\}$,
\begin{eqnarray}
\hspace{-5mm}P(X_1>k_1,X_2>k_2,\ldots,X_n>k_n)&=&\prod_{\emptyset\neq I\subset \{1,2,\ldots,n\}}\theta_{I}^{\max\{k_i, i\in I\}+1}  \label{sfMVG} \\  
&=& \prod\limits_{i=1}^n\prod\limits_{1\leq j_1<\ldots<j_i\leq n}\theta_{\{j_1,\ldots,j_i\}}^{\max\{k_{j_1},\ldots,k_{j_i}\}+1}. \nonumber
\end{eqnarray}
\item[(ii)] $X_{1:n}\sim ge\left(1-\prod\limits_{\emptyset\neq I\subset \{1,2,\ldots,n\}}\theta_I\right).$
\item[(iii)]If $1\leq s\leq n$, $1\leq l_1<l_2<\cdots<l_s\leq n$, and $\tilde{J}$ denotes the family of all subsets (along with the empty set) of the set $\{1,2,\ldots,n\}\setminus\{l_1,\ldots,l_s\}$, then the random vector $(X_{l_1},X_{l_2},\ldots,X_{l_s})$ has the MVG distribution with parameters
\begin{equation*}
\hat{\theta}_I=\prod\limits_{\tilde{I}\in\tilde{J}}\theta_{I\cup\tilde{I}}, \quad \emptyset\neq I\subset\{l_1,l_2,\ldots,l_s\}.
\end{equation*}
\item[(iv)] Under the assumptions of Part (iii) we have 
\begin{equation}
\label{MinSubset}
X_{1:\{l_1,l_2,\ldots,l_s\}} \sim ge\left(
1-\prod_{I\subset\{1,\ldots,n\},  I\cap\{l_1,\ldots,l_s\}\neq \emptyset} \theta_I \right)
\end{equation}
and, in consequence, for $p=1,2,\ldots$,
\begin{equation}
\label{factorialMin}
E\left(X_{1:\{l_1,l_2,\ldots,l_s\}}\right)_p
=p!\left(\frac{\theta}{1-\theta}\right)^{p},   
\end{equation}
where $ \theta=\prod_{I\subset\{1,\ldots,n\},  I\cap\{l_1,\ldots,l_s\}\neq \emptyset} 
  \theta_I$.
\end{itemize}
\end{theorem}

\begin{proof}
(i) Let $J$ be as in Definition \ref{MVG}. Then, for $k_{1},  \ldots,k_{n}\in\{-1,0,1,\ldots\}$, 
\begin{eqnarray*}
P(X_1>k_1,\ldots,X_n>k_n)&=&P\left(\min\{M_I: I\ni i\}>k_i, i=1,2,\ldots,n\right)   \\
&=&P\left(M_I>\max\{k_i,i\in I\} \textrm{ for every  } I\in J\right)  \\
&=&\prod\limits_{I\in J}P\left(M_I>\max\{k_i,i\in I\}\right)  \\
&=&\prod\limits_{I\in J}\theta_I^{\max\{k_i,i\in I\}+1} \\
&=&\prod\limits_{i=1}^n\prod\limits_{1\leq j_1<\ldots<j_i\leq n}\theta_{\{j_1,\ldots,j_i\}}^{\max\{k_{j_1},\ldots,k_{j_i}\}+1}
\end{eqnarray*}
as required.

\medskip\noindent 
(ii) Since $P(X_{1:n}>k)=P(X_1>k,X_2>k,\ldots,X_n>k)$, 
(\ref{sfMVG})  immediately implies  
\begin{eqnarray*}
P(X_{1:n}>k)&=&\prod_{\emptyset\neq I\subset \{1,2,\ldots,n\}}\theta_I^{\max\{k,\ldots,k\}+1} \\
&=&\left(\prod_{\emptyset\neq I\subset \{1,2,\ldots,n\}}\theta_I\right)^{k+1},\quad   k=-1,0,1,\ldots,
\end{eqnarray*}
which gives the desired conclusion. 

\medskip\noindent 
(iii) Let $k_i=-1$ if $i\notin\{l_1,\ldots,l_s\}$. Then, for $k_{l_1},  \ldots,k_{l_s}\in\{-1,0,1,\ldots\}$, 
\begin{eqnarray}
P\left(X_{l_1}>k_{l_1},\ldots,X_{l_s}>k_{l_s}\right)&=& P\left(X_i>k_i, i=1,2,\ldots,n\right)\nonumber\\
&=&\prod_{\emptyset\neq I\subset\{1,2,\ldots,n\}}\theta_I^{\max\{k_i,i\in I\}+1}   \nonumber\\
&=&\prod_{I\subset\{1,2,\ldots,n\},I\cap\{l_1,\ldots,l_s\}\neq \emptyset}\
\theta_I^{\max\{k_i, i\in I\cap\{l_1,\ldots,l_s\} \}+1}\nonumber\\
&=&\prod_{\emptyset\neq I\subset\{l_1,\ldots,l_s\}}\left(\prod\limits_{\tilde{I}\in\tilde{J}}\theta_{I\cup\tilde{I}}\right)^{\max\{k_{i}, i\in I\}+1}, \label{rozk_brzeg}
\end{eqnarray}
where the second equality is a consequence of (\ref{sfMVG}). Comparing (\ref{rozk_brzeg}) with (\ref{sfMVG}) gives the assertion of (iii).

\medskip\noindent 
(iv) By (ii) and (iii), $X_{1:\{l_1,l_2,\ldots,l_s\}}$ has the geometric distribution with parameter $1-\theta$, where
$$\theta=\prod\limits_{\emptyset\neq I\subset \{l_1,l_2,\ldots,l_s\}}\hat{\theta}_I
=\prod\limits_{\emptyset\neq I\subset \{l_1,l_2,\ldots,l_s\}}  \prod\limits_{\tilde{I}\in\tilde{J}}\theta_{I\cup\tilde{I}}
=\prod_{I\subset\{1,\ldots,n\},  I\cap\{l_1,\ldots,l_s\}\neq \emptyset} \theta_I$$
as claimed in  (\ref{MinSubset}). Relation (\ref{factorialMin}) is an immediate consequence of the well known formula for factorial moments of the geometric distribution.

\end{proof}

In particular, Theorem \ref{MVG_properties} (iii) asserts that univariate marginal distributions of the MVG distribution are geometric. More precisely, if the random vector $(X_1,X_2,\ldots,X_n)$ has the MVG distribution with parameters $\theta_I\in [0,1]$, $\emptyset\neq I\subset \{1,2,\ldots,n\}$, then 
$$X_i\sim ge\left(1-\prod\limits_{\{i\}\subset I\subset\{1,2,\ldots,n\}}\theta_I  \right), \quad i=1,2,\ldots,n.$$
 If moreover, $\theta_I=1$ for all $\emptyset\neq I\subset \{1,2,\ldots,n\}$ except singletons, then $X_1,X_2,\ldots,X_n$ are independent and $X_i\sim ge\left(1-\theta_{\{i\}}\right)$, $i=1,2,\ldots,n$.

Theorem \ref{MVG_properties} (iv)  provides  formulas for factorial moments of the smallest order statistic from  random vector $(X_1,X_2,\ldots,X_n)$ with MVG distribution. To extend these formulas to larger  order statistics we will use the following result by Balakrishnan et. al. \cite{BBM92}. For some other interesting relations for cumulative distribution functions of order statistics we refer the reader to \cite{E13}.

\begin{theorem}
Let the  random vector $(X_1,X_2,\ldots,X_n)$ have any joint distribution and $1\leq r\leq n$. By $F_{r:n}$ and $F_{1:S}$, $S\subset\{1,2,\ldots,n\}$,  denote the cumulative distribution functions of $X_{r:n}$ and $X_{1:S}$, respectively. Then, for all $x$,
\begin{equation}
\label{recc}
F_{r:n}(x)=\sum_{j=0}^{r-1} (-1)^{r-1-j} {{n-j-1} \choose {n-r}} \sum_{S \subset \{1,\ldots,n\}, |S|=n-j} F_{1:S}(x),
\end{equation}
where $|S|$ stands for the number of elements of the set $S$.
\end{theorem}

Relation (\ref{recc}) can be rewritten in terms of probability mass functions and hence in terms of factorial moments, provided they exist. For $1\leq r\leq n$ and $p=1,2,\ldots$, we get 
\begin{equation}
\label{reccfac}
 E\left(X_{r:n}\right)_p=\sum_{j=0}^{r-1} (-1)^{r-1-j} 
{{n-j-1} \choose {n-r}}\sum_{S \subset \{1,\ldots,n\},|S|=n-j} E\left(X_{1:S}\right)_p.
\end{equation}
Combining Theorem \ref{MVG_properties} (iv)   with (\ref{reccfac}) we obtain the main result of this section.

\begin{theorem}
\label{MVGmoments}
Under the assumption of Theorem \ref{MVG_properties}, for $1\leq r\leq n$ and $p=1,2,\ldots$, 
\begin{equation}
\label{MVGfac}
E\left(X_{r:n}\right)_p 
=p!\sum_{j=0}^{r-1} (-1)^{r-1-j} {{n-j-1} \choose {n-r}} S_{j,p},   
\end{equation}
where
$$S_{0,p}=\left(\frac{1}{1-\theta_{all}}-1  \right)^p$$
and, for $1\leq j\leq n-1$,
$$S_{j,p}=\sum_{\{s_1,\ldots,s_{j}\}\subset\{1,\ldots,n\}} \left(\frac{1}{1-\frac{\theta_{all}}{\prod\limits_{\emptyset\neq I\subset \{s_1,\ldots,s_{j}\}}\theta_I}}-1  \right)^p$$
with 
\begin{equation}
\label{teta.all}
\theta_{all}=\prod\limits_{\emptyset\neq I\subset \{1,2,\ldots,n\}}\theta_I.
\end{equation}
In particular,
\begin{equation}
\label{MVGexp}
E\left(X_{r:n}\right)=\sum_{j=0}^{r-1} (-1)^{r-1-j} {{n-j-1} \choose {n-r}} S_{j,1}
\end{equation}
and
\begin{eqnarray}
\label{MVGvar}
&&\hspace{-7mm}Var\left(X_{r:n}\right) =E\left(X_{r:n}(X_{r:n}-1) \right)+E\left(X_{r:n}\right)-\big(E\left(X_{r:n}\right)\big)^2     \nonumber \\
&=& 2\sum_{j=0}^{r-1} (-1)^{r-1-j} {{n-j-1} \choose {n-r}} S_{j,2}+E\left(X_{r:n}\right)\left(1-E\left(X_{r:n}\right)\right).
\end{eqnarray}
\end{theorem}
\begin{proof}
For $0\leq j\leq n-1$,
\begin{eqnarray}
&&\sum_{S \subset \{1,\ldots,n\},|S|=n-j} E\left(X_{1:S}\right)_p=\sum_{\{s_1,\ldots,s_{n\!-\!j}\}\subset\{1,\ldots,n\}}
E\left(X_{1:\{s_1,\ldots,s_{n-j}\}}\right)_p  \nonumber \\
&=&p!\sum_{\{s_1,\ldots,s_{n-j}\}\subset\{1,\ldots,n\}}  
\left(\frac{1}{1-\prod\limits_{I\subset \{1,\ldots,n\},I\cap  \{s_1,\ldots,s_{n-j}\}\neq\emptyset}\theta_I}-1  \right)^p. 
  \label{thmMVG1}
\end{eqnarray}
by Theorem \ref{MVG_properties} (iv). 
But, for $j=0$ and $\{s_1,\ldots,s_{n}\}\subset \{1,\ldots,n\}$,
\begin{equation}
\label{thmMVG2}
\prod\limits_{I\subset \{1,\ldots,n\},I\cap  \{s_1,\ldots,s_{n-j}\}\neq\emptyset}\theta_I = 
\prod\limits_{\emptyset\neq I\subset \{1,\ldots,n\}}\theta_I=\theta_{all},
\end{equation}
while for $1\leq j\leq n-1$ and $\{s_1,\ldots,s_{n-j}\}\subset \{1,\ldots,n\}$,
\begin{eqnarray}
\prod\limits_{I\subset \{1,\ldots,n\},I\cap  \{s_1,\ldots,s_{n-j}\}\neq\emptyset}\theta_I
&=&\frac{\prod\limits_{\emptyset\neq I\subset \{1,\ldots,n\}}\theta_I}{
\prod\limits_{\emptyset\neq I\subset \{1,\ldots,n\}\setminus \{s_1,\ldots,s_{n-j}\}}\theta_I} \nonumber \\
&=&\frac{\theta_{all}}{\prod\limits_{\emptyset\neq I\subset \{1,\ldots,n\}\setminus \{s_1,\ldots,s_{n-j}\}}\theta_I}. \label{thmMVG3}
\end{eqnarray}
Substituting (\ref{thmMVG2}) and (\ref{thmMVG3}) into  (\ref{thmMVG1}), and noticing that, for $1\leq j\leq~n-~1$,
\begin{eqnarray*}
\sum_{\{s_1,\ldots,s_{n-j}\}\subset\{1,\ldots,n\}}  
\left(\frac{1}{1-\frac{\theta_{all}}{\prod\limits_{\emptyset\neq I\subset \{1,\ldots,n\}\setminus \{s_1,\ldots,s_{n-j}\}}\theta_I}}-1  \right)^p && \\
=\sum_{\{s_1,\ldots,s_{j}\}\subset\{1,\ldots,n\}} \left(\frac{1}{1-\frac{\theta_{all}}{\prod\limits_{\emptyset\neq I\subset \{s_1,\ldots,s_{j}\}}\theta_I}}-1  \right)^p=S_{j,p}, &&
\end{eqnarray*}
we get 
\begin{eqnarray*}
&&\sum_{\{s_1,\ldots,s_{n-j}\}\subset\{1,\ldots,n\}}
E\left(X_{1:\{s_1,\ldots,s_{n-j}\}}\right)_p \\
&=&\left\{
\begin{array}{lcl}
p!\left(\frac{1}{1-\theta_{all}}-1  \right)^p=p!S_{0,p} & \hbox{ if } &  j= 0 \\
p!S_{j,p}  & \hbox{ if } &  1\leq j\leq n-1
\end{array}
\right.. 
\end{eqnarray*}
Now application of (\ref{reccfac}) finishes the proof.
\end{proof}

Applying (\ref{MVGexp}) and  (\ref{MVGvar}) to the case when $X_1, X_2, \ldots, X_n$ are independent, i.e., $\theta_I=1$ for all $\emptyset\neq I\subset \{1,2,\ldots,n\}$ except singletons, we recover known formulas for the expectation and variance of the $r$th order statistic from INID geometric RV's \citep[formulas (4.11) and (4.13)]{DD18}.

Furthermore, under the additional assumption that the random vector $(X_1, X_2, \ldots, X_n)$ is exchangeable, Theorem \ref{MVGmoments} takes on the following simpler form. In its formulation  we  adopt the convention that 
\begin{equation}
\label{konw}
{{j} \choose {s}}=0 \; \hbox{  if } \; s>j\geq 1.
\end{equation}
\begin{corollary}
\label{exchMVGmoments}
Let the random vector $(X_1,X_2,\ldots,X_n)$ be exchangeable and have the MVG distribution with parameters $\theta_I\in [0,1]$, $\emptyset\neq I\subset \{1,2,\ldots,n\}$, where
\begin{equation}
\label{cMVGexch}
\theta_{\{i_1,\ldots,i_s\}}=\theta_s\in[0,1] \quad \hbox{ for all } 1\leq s\leq n \hbox{ and } 1\leq i_1<\cdots<i_s\leq n
\end{equation}
and at least one $\theta_s$, $1\leq s\leq n$, be not equal to 1. Then, for $1\leq r\leq n$ and $p=1,2,\ldots$,
$$
 E\left(X_{r:n}\right)_p   
=p!\sum_{j=0}^{r-1} (-1)^{r-1-j} {{n-j-1} \choose {n-r}}{{n} \choose {j}}\left(\frac{1}{1-\prod\limits_{s=1}^n \theta_s^{{{n} \choose {s}}-{{j} \choose {s}}}} -1\right)^p. 
$$
In particular,
$$
E\left(X_{r:n}\right)=\sum_{j=0}^{r-1} (-1)^{r-1-j} {{n-j-1} \choose {n-r}}{{n} \choose {j}}\left(\frac{1}{1-\prod\limits_{s=1}^n \theta_s^{{{n} \choose {s}}-{{j} \choose {s}}}} -1\right)
$$
and
\begin{eqnarray*}
Var\left(X_{r:n}\right)&=&2\sum_{j=0}^{r-1} (-1)^{r-1-j} {{n-j-1} \choose {n-r}}{{n} \choose {j}}\left(\frac{1}{1-\prod\limits_{s=1}^n \theta_s^{{{n} \choose {s}}-{{j} \choose {s}}}} -1\right)^2 \\
&+&E\left(X_{r:n}\right)\left(1-E\left(X_{r:n}\right)\right).
\end{eqnarray*}
\end{corollary}
\begin{proof}
Condition (\ref{cMVGexch}) guarantees that the  random vector $(X_1,X_2,\ldots,X_n)$ is indeed exchangeable. If this is the case, then
\begin{eqnarray*}
\theta_{all}&=&\prod\limits_{\emptyset\neq I\subset \{1,2,\ldots,n\}}\theta_I=
\prod_{s=1}^n\prod_{\{i_1,\ldots,i_{s}\}\subset\{1,\ldots,n\}}\theta_{\{i_1,\ldots,i_{s}\}} \\
&=&\prod_{s=1}^n\prod_{\{i_1,\ldots,i_{s}\}\subset\{1,\ldots,n\}}\theta_s=\prod_{s=1}^n\theta_s^{{n} \choose {s}}
\end{eqnarray*}
and, for $\{s_1,\ldots,s_{j}\}\subset\{1,\ldots,n\}$,
$$
\prod\limits_{\emptyset\neq I\subset \{s_1,\ldots,s_j\}}\theta_I=
\prod_{s=1}^j\theta_s^{{j} \choose {s}}=
\prod_{s=1}^n\theta_s^{{j} \choose {s}},
$$
by (\ref{konw}).
 Consequently, using notation of Theorem \ref{MVGmoments}, we have
\begin{equation}
\label{coMVGexch1}
S_{0,p}={{n} \choose {0}}\left(\frac{1}{1-\prod\limits_{s=1}^n\theta_s^{{n} \choose {s}}}-1  \right)^p
\end{equation}
and, for $1\leq j\leq n-1$,
\begin{eqnarray}
S_{j,p}&=&\sum_{\{s_1,\ldots,s_{j}\}\subset\{1,\ldots,n\}} \left(\frac{1}{1-\prod\limits_{s=1}^n\theta_s^{{{n} \choose {s}}-{{j} \choose {s}}}}-1  \right)^p  \nonumber  \\
&=&{{n} \choose {j}} \left(\frac{1}{1-\prod\limits_{s=1}^n\theta_s^{{{n} \choose {s}}-{{j} \choose {s}}}}-1  \right)^p.
\label{coMVGexch2}
\end{eqnarray}
Substituting (\ref{coMVGexch1}) and (\ref{coMVGexch2}) into (\ref{MVGfac}) completes the proof.
\end{proof}

We finish this section with Tables \ref{Table6} and \ref{Table7} presenting means and variances (in brackets) of all order statistics from some MVG samples of size $10$ which are non-exchangeable and exchangeable, respectively. They were obtained using Theorem \ref{MVGmoments} and Corollary \ref{exchMVGmoments}.

\begin{landscape}
\begin{table}[ht]
\footnotesize
\caption{Mean and variance (in brackets) of $X_{r:10}$ from $(X_1,\ldots,X_{10})$ with MVG distribution with parameters $\theta_I$, $\emptyset\neq I\subset\{1,\dots,10\}$ (not listed parameters are assumed to be equal to 1)}
\begin{tabular}{|l||c|c|c|c|c|c|c|c|c|c||}
\hline
$\theta_I\backslash r$ &1 & 2&3 &4 &5 &6 &7 &8 &9 &10\\ \hline\hline
$\theta_{\{1\}}=\cdots=\theta_{\{8\}}=0.9$ & 0.375 & 1.138  & 2.110  &3.239  &4.563  & 6.157 &8.149  &10.784  & 14.644 &  21.851\\
$\theta_{\{9\}}=\theta_{\{10\}}=0.8$& (0.516)  & (1.407)    & (2.456) &(3.876)  & (5.978)&(9.271)&(14.827)&(25.311)  &(49.390)  & (137.343) \\
$\theta_{\{1,\ldots,10\}}=0.99$&  &  &  &  &  &  &  &  &  &  \\\hline
$\theta_{\{1\}}=\cdots=\theta_{\{8\}}=0.9$ & 0.213 & 0.583 & 1.115 & 1.760 & 2.525 &3.450 &4.614  & 6.184 & 8.566 & 13.406 \\
$\theta_{\{9\}}=\theta_{\{10\}}=0.8$&        (0.258) & (0.681) & (1.194) &(1.787)&(2.546)&(3.623)&(5.287)&(8.202)&(14.656)  & (40.025) \\
$\theta_{\{i,j\}}=0.99, 1\leq i<j\leq 10$&  &  &  &  &  &  &  &  &  &  \\\hline
$\theta_{\{1\}}=\cdots=\theta_{\{8\}}=0.9$ & 0.336 & 0.997 &1.850  & 2.860 & 4.061 & 5.535 & 7.424 & 10.016 & 14.030 & 22.350 \\
$\theta_{\{9\}}=\theta_{\{10\}}=0.8$ & (0.449) & (1.221) & (2.121) & (3.258) & (4.849) & (7.238) & (11.153) & (18.502) & (35.978) &(109.293)\\
$\theta_{\{1,j\}}=0.99, 2\leq j\leq 10$&  &  &  &  &  &  &  &  &  &  \\\hline
$\theta_{\{1\}}=\cdots=\theta_{\{8\}}=0.9$ & 0.314 & 0.919 & 1.674 & 2.524 & 3.478 & 4.565 & 5.835 & 7.372 & 9.344 &  12.209\\
$\theta_{\{9\}}=\theta_{\{10\}}=0.8$& (0.413) & (1.118) & (1.960) & (3.088) & (4.779) & (7.469) & (11.993) & (20.185) & (36.868) & (80.375) \\
$\theta_{\{1,j\}}=0.99, 2\leq j\leq 10$&  &  &  &  &  &  &  &  &  &  \\
$\theta_{\{1,\ldots,10\}}=0.95$&  &  &  &  &  &  &  &  &  &  \\\hline
\end{tabular}
\label{Table6}
\end{table}

\begin{table}[ht]
\footnotesize
\caption{Mean and variance (in brackets) of $X_{r:10}$ from $(X_1,\ldots,X_{10})$ with exchangeable MVG distribution with parameters $\theta_s$, $s=1,2,\ldots,10$, defined in (\ref{cMVGexch}) (not listed parameters are assumed to be equal to 1)}
\begin{tabular}{|l||c|c|c|c|c|c|c|c|c|c||}
\hline
$\theta_s\backslash r$ &1 & 2&3 &4 &5 &6 &7 &8 &9 &10\\ \hline\hline
$\theta_{1}=0.9$& 0.285 & 0.705 & 1.303 & 2.008 & 2.839 & 3.835 &5.080  & 6.740 & 9.229 &14.208  \\
$\theta_{2}=0.99$& (0.366) & (0.885)&(1.499)&(2.208)&(3.101)&(4.338)&(6.197)&(9.366)&(16.184)&(42.216)  \\\hline
$\theta_{1}=0.9$& 0.036 &0.077  & 0.211 & 0.387 & 0.656 & 0.992 & 1.420 & 1.984 & 2.828 & 4.515 \\
$\theta_{2}=0.95$&(0.037)& (0.080) &(0.203)&(0.345)&(0.513)&(0.694)&(0.927)&(1.317)&(2.150)& (5.244)  \\\hline
$\theta_{1}=0.9$& 0.036 & 0.077 &0.209   &0.382  &0.648  &0.979  &1.398  &1.948  &2.764  & 4.367 \\
$\theta_{2}=0.95$&(0.037)&(0.080)&(0.201)&(0.341)&(0.509)&(0.690)&(0.926)&(1.324)& (2.177) & (5.293) \\
$\theta_{10}=0.99$&  &  &  &  &  &  &  &  &  &  \\\hline
$\theta_{2}=0.95$& 0.110 & 0.110 & 0.403 & 0.560 & 0.948 & 1.332 & 1.857 & 2.540 & 3.567 & 5.619 \\
                 & (0.123) & (0.123) & (0.398) & (0.546) & (0.791) & (1.065) & (1.425) & (2.040) & (3.312) & (7.967) \\\hline
$\theta_{8}=0.95$& 0.110 & 0.110 & 0.110 & 0.110 & 0.110 & 0.110 & 0.110 & 0.110 & 0.403 & 0.587 \\
                 & (0.123) & (0.123) & (0.123) & (0.123) & (0.123) & (0.123) & (0.123) & (0.123) & (0.398) & (0.592) \\\hline
\end{tabular}
\label{Table7}
\end{table}
\end{landscape}

\section{Moments of lifetimes of coherent systems}
\label{section4}

In this section we establish formulas that are  convenient for numerical evaluation of moments of lifetimes of coherent systems  operating in discrete time.

First observe that from results given in Sections \ref{section2} and \ref{section3} we immediately obtain  moments of lifetimes of two important types of technical structures, namely of $k$-out-of-$n:F$ and $k$-out-of-$n:G$ systems. To see this, let $X_1,X_2,\ldots,X_n$ be random lifetimes of  $n$ items and $T_{k,n:F}$  ($T_{k,n:G}$) denote the lifetime of the $k$-out-of-$n:F$  ($k$-out-of-$n:G$) system built up from these items. Since a $k$-out-of-$n:F$ system fails when at least $k$ of its components are broken, and $k$-out-of-$n:G$ system works as long as at least $k$ of its components are working we have, for any joint distribution of $(X_1,X_2,\ldots,X_n)$ and $1\leq k\leq n$,
$$
T_{k,n:F}=X_{k:n} \hbox{ and } T_{k,n:G}=X_{n-k+1:n}.
$$
Therefore using Theorem \ref{th1} (\ref{th2})
we can find moments (their approximate values) of $T_{k,n:F}$  and $T_{k,n:G}$  whenever $X_1,X_2,\ldots,X_n$ are discrete RV's taking values in finite (infinite) subsets of the set of non-negative integers. In particular, Corollaries \ref{th2c1} and \ref{th2c2} 
enable us to compute approximations of these moments when the univariate marginal distributions of $(X_1,X_2,\ldots,X_n)$ are Poisson or negative binomial. Theorem \ref{MVGmoments} along with Corollary \ref{exchMVGmoments}  give moments of $T_{k,n:F}$  and $T_{k,n:G}$ in the case when the joint distribution of component lifetimes is MVG. Furthermore, values in Tables \ref{Table1}-\ref{Table7}
can be interpreted as means, second raw moments and/or variances of $T_{k,n:F}$ (for $r=k$) and $T_{k,n:G}$ (for $r=n-k+1$).
\begin{example}
Water is supplied to a factory by $10$ pipes. Every morning of a working day valves are opened to provide water to the factory. There is one main valve which opens general water flow to all the pipes and $10$ valves which open water flow into separate pipes. The factory can work as long as at least $k$ ($1\leq k\leq 10$) pipes supply water. If the probability that the main valve will not break down during opening is $\theta_{\{1,\ldots,10\}}$, the probability that the valve of the $i$th pipe will not break down during opening is $\theta_{\{i\}}$, $i=1,\ldots,10$, all the valves work independently and are not repaired, then the system supplying water to the factory is a $k$-out-of-$10:G$  system with components having the MVG distribution with parameters $\theta_I$, $\emptyset\neq I\subset\{1,\dots,10\}$, where $\theta_I=1$ for $I\notin\{\{1\},\ldots,\{10\},\{1,\ldots,10\}\}$. Results of Section \ref{section3}  allow to find moments of the random number of working days up to a failure of this system. 
In particular, taking $r=n-k+1=11-k$ in the first row of Table \ref{Table6}, one obtains means and variances of this number for $k=1,\ldots,10$  when $\theta_{\{1\}}=\cdots=\theta_{\{8\}}=0.9$, $\theta_{\{9\}}=\theta_{\{10\}}=0.8$ and $\theta_{\{1,\ldots,10\}}=0.95$.
\end{example}

If we want to find moments of lifetimes of coherent systems other  than $k$-out-of-$n$ structures we need more effort. The rest of this section is denoted to presenting methods of doing this. We start with recalling relevant concepts and facts.

Let us consider a coherent system $S$ consisting of $n$ components numbered $1,2,\ldots,n$. We say that $P\subset\{1,2,\ldots,n\}$ is a path set of $S$ if system $S$  functions when all the elements with indices in $P$ work. Similarly, $C\subset\{1,2,\ldots,n\}$ is called a cut set of $S$ if system $S$ fails whenever all the elements with indices in $C$ break down. A path (cut) set is said to be minimal if it does not contain any strict subset being a path (cut) set.

Using the concepts of minimal path and cut sets we can write a useful representation for the lifetime of a coherent system. If the coherent system $S$ has $s$ minimal path sets $P_1,P_2,\ldots,P_s$ and $v$ minimal cut sets $C_1,C_2,\ldots,C_v$, then 
\begin{equation}
\label{repMPCs}
T=\max\limits_{1\leq j\leq s}\min\limits_{i\in P_j}X_i=\min\limits_{1\leq j\leq v}\max\limits_{i\in C_j}X_i,
\end{equation}
where $T$ denotes the lifetime of system $S$ and $X_i$ is the lifetime of its $i$th component, $i=1,\ldots,n$; see \cite[p 13]{BP75}. 
Applying (\ref{repMPCs}) and the inclusion-exclusion formula we get expressions for the survival function of T (see \cite[pp. 25-26]{BP75} and \cite{NRS07})
\begin{eqnarray}
\hspace{-4mm} P(T>t)&=&\sum\limits_{j=1}^s(-1)^{j+1}\sum\limits_{\{k_1,\ldots,k_j\}\subset\{1,\ldots, s\}}P\left(X_{1:\bigcup_{l=1}^jP_{k_l}}>t\right) \label{sf1} \\
&=&1-\sum\limits_{j=1}^v(-1)^{j+1}\sum\limits_{\{k_1,\ldots,k_j\}\subset\{1,\ldots, v\}}P\left(X_{\left|\bigcup_{l=1}^jC_{k_l}\right| :\bigcup_{l=1}^jC_{k_l}}\leq t\right) \nonumber \\
&=&  \sum\limits_{j=1}^v(-1)^{j+1}\sum\limits_{\{k_1,\ldots,k_j\}\subset\{1,\ldots, v\}}P\left(X_{\left|\bigcup_{l=1}^jC_{k_l}\right| :\bigcup_{l=1}^jC_{k_l}}> t\right), \label{sf2}
\end{eqnarray}
where,   for $A=\{i_1,\ldots,i_m\}\subset\{1,\ldots,n\}$, $1\leq m\leq n$, and $1\leq r\leq m$,
$X_{r:A}$ denotes the $r$th order statistic from $X_{i_1},\ldots,X_{i_m}$, and as in the previous section $|A|$ stands for the number of elements in $A$.
Under the assumption that the random vector $(X_1,X_2,\ldots,X_n)$ is exchangeable,  (\ref{sf1}) and (\ref{sf2}) simplify considerably. Indeed, then the distribution of $X_{1:\bigcup_{l=1}^jP_{k_l}}$ and $X_{\left|\bigcup_{l=1}^jC_{k_l}\right| :\bigcup_{l=1}^jC_{k_l}}$ depends only on the number of elements  in $\bigcup_{l=1}^jP_{k_l}$ and $\bigcup_{l=1}^jC_{k_l}$, respectively, and  we get
\begin{equation}
\label{sfexch}
P(T>t)=\sum_{i=1}^n\alpha_iP(X_{1:i}>t)=\sum_{i=1}^n\beta_iP(X_{i:i}>t),
\end{equation}
where $\alpha_i, \beta_i$, $i=1,2,\ldots,n$, are real numbers depending on the structure of the coherent system but not on the distribution of $(X_1,X_2,\ldots,X_n)$, and satisfying $\sum_{i=1}^n\alpha_i=\sum_{i=1}^n\beta_i=1$. The vectors $\mathbf{a}=(\alpha_1,\ldots,\alpha_n)$ and $\mathbf{b}=(\beta_1,\ldots,\beta_n)$ are called minimal and maximal signatures, respectively; see \cite{NRS07}.

Now we are ready to state and prove the main result of this section providing formulas for raw and factorial moments of lifetimes of  coherent systems  operating in discrete time. In the sequel, $\mathbb{I}(\cdot)$  denotes the indicator function, i.e., $\mathbb{I}(B)=1$ if $B$ is true and $\mathbb{I}(B)=0$ otherwise.
\begin{theorem}
\label{mCS}
Let $T$ be the lifetime of a coherent system with component lifetimes  $X_1,X_2,\ldots,X_n$ and with minimal path and cut sets $P_1,P_2,\ldots,P_s$ and  $C_1,C_2,\ldots,C_v$, respectively.
\begin{itemize}
\item[(i)]  If $X_i$ takes values in the set of non-negative integers, $i=1,2,\ldots,n$, then, for $p=1,2,\ldots$,
\begin{eqnarray}
ET^p&=&\sum\limits_{m=0}^\infty \big((m+1)^p-m^p\big)  \sum_{i=1}^{n} \sum\limits_{\{k_1,\ldots,k_i\}\subset\{1,\ldots, n\}}
\alpha_{\{k_1,\ldots,k_i\}} \nonumber  \\
&&\times P\left(\bigcap\limits_{l=1}^i\{X_{k_l}>m\}\right)  \label{rmCS1} \\
&=&\sum\limits_{m=0}^\infty \big((m+1)^p-m^p\big)  \sum_{i=1}^{n} \sum\limits_{\{k_1,\ldots,k_i\}\subset\{1,\ldots, n\}}
\beta_{\{k_1,\ldots,k_i\}} \nonumber  \\
&&\times \left\{1-P\left(\bigcap\limits_{l=1}^i\{X_{k_l}\leq m\}\right)\right\},  \label{rmCS1A}
\end{eqnarray}
where, for $\{k_1,\ldots,k_i\}\subset\{1,\ldots,n\}$,
\begin{equation}
\label{alfy}
\alpha_{\{k_1,\ldots,k_i\}}=\sum\limits_{j=1}^s(-1)^{j+1}\sum\limits_{\{s_1,\ldots,s_j\}\subset\{1,\ldots, s\}}\mathbb{I}\left(\bigcup_{l=1}^jP_{s_l}=\{k_1,\ldots,k_i\}\right) 
\end{equation}
and
\begin{equation*}
\beta_{\{k_1,\ldots,k_i\}}=\sum\limits_{j=1}^v(-1)^{j+1}\sum\limits_{\{s_1,\ldots,s_j\}\subset\{1,\ldots, v\}}\mathbb{I}\left(\bigcup_{l=1}^jC_{s_l}=\{k_1,\ldots,k_i\}\right) .
\end{equation*}
\item[(ii)] Under the hypothesis of part (i), if moreover we assume that $j_{max(m)}$ defined in (\ref{jmaxm}) does not depend on $m$, denote it by $j_0$ and require that for a fixed $p\in\{1,2,\ldots\}$, $EX_{j_0}^p<\infty$, then 
\begin{eqnarray}
ET^p&\approx&\sum\limits_{m=0}^{M_0} \big((m+1)^p-m^p\big)  \sum_{i=1}^{n} \sum\limits_{\{k_1,\ldots,k_i\}\subset\{1,\ldots, n\}}
\alpha_{\{k_1,\ldots,k_i\}} \nonumber  \\
&&\times P\left(\bigcap\limits_{l=1}^i\{X_{k_l}>m\}\right)  \label{rmCSapp}
\end{eqnarray}
and
\begin{eqnarray}
ET^p&\approx&\sum\limits_{m=0}^{\bar{M}_0} \big((m+1)^p-m^p\big)  \sum_{i=1}^{n} \sum\limits_{\{k_1,\ldots,k_i\}\subset\{1,\ldots, n\}}\beta_{\{k_1,\ldots,k_i\}} \nonumber  \\
&&\times \left\{1-P\left(\bigcap\limits_{l=1}^i\{X_{k_l}\leq m\}\right)\right\}
 \label{rmCSappA}
\end{eqnarray}
introduce an error not greater than $d>0$, provided that
\begin{equation}
\label{M0CS}
 \sum_{x=M_0+2}^{\infty}x^pP(X_{j_0}=x)\leq d \left(\sum_{i=1}^{n}\sum\limits_{\{k_1,\ldots,k_i\}\subset\{1,\ldots, n\}}\alpha_{\{k_1,\ldots,k_i\}}^+\right)^{-1}
\end{equation}
and
\begin{equation*}
 \sum_{x=\bar{M}_0+2}^{\infty}x^pP(X_{j_0}=x)\leq \frac{d}{2^n-1} \left(\sum_{i=1}^{n}\sum\limits_{\{k_1,\ldots,k_i\}\subset\{1,\ldots, n\}}\beta_{\{k_1,\ldots,k_i\}}^+\right)^{-1}
\end{equation*}
respectively, where
$$
\alpha_{\{k_1,\ldots,k_i\}}^+=
\left\{
\begin{array}{ll}
\alpha_{\{k_1,\ldots,k_i\}} & \hbox{if }  \;  \alpha_{\{k_1,\ldots,k_i\}}>0 \\
0  & \hbox{otherwise} 
\end{array}
\right.
$$
and 
$$
\beta_{\{k_1,\ldots,k_i\}}^+=
\left\{
\begin{array}{ll}
\beta_{\{k_1,\ldots,k_i\}} & \hbox{if }  \;  \beta_{\{k_1,\ldots,k_i\}}>0 \\
0  & \hbox{otherwise} 
\end{array}
\right..
$$
\item[(iii)] If $X_1,X_2,\ldots,X_n$ are non-negative RV's and for a fixed $p\in\{1,2,\ldots\}$ we have
\begin{equation}
\label{thmCSiii}
E\left(X_{1:\{k_1,\ldots,k_i\} }^p \right)<\infty \; \hbox{ for } \emptyset\neq \{k_1,\ldots,k_i\}\subset \{1,\ldots, n\}
\end{equation}
$$
\left( E\left(X_{i:\{k_1,\ldots,k_i\} }^p \right)<\infty \; \hbox{ for } \emptyset\neq \{k_1,\ldots,k_i\}\subset \{1,\ldots, n\} \right),
$$
then
\begin{equation}
ET^p = \sum_{i=1}^{n} \sum\limits_{\{k_1,\ldots,k_i\}\subset\{1,\ldots, n\}}\alpha_{\{k_1,\ldots,k_i\}}E\left(X_{1:\{k_1,\ldots,k_i\} }^p \right) 
 \label{rmCS2} 
\end{equation}
\begin{equation}
\left(ET^p  = \sum_{i=1}^{n} \sum\limits_{\{k_1,\ldots,k_i\}\subset\{1,\ldots, n\}}\beta_{\{k_1,\ldots,k_i\}}E\left(X_{i:\{k_1,\ldots,k_i\} }^p \right) \right)
 \label{rmCS2A}
\end{equation}
and
\begin{equation}
E(T)_p = \sum_{i=1}^{n} \sum\limits_{\{k_1,\ldots,k_i\}\subset\{1,\ldots, n\}}\alpha_{\{k_1,\ldots,k_i\}} 
 E\left(X_{1:\{k_1,\ldots,k_i\} }\right)_p  \label{fmCS} 
\end{equation}
\begin{equation}
\left( E(T)_p = \sum_{i=1}^{n} \sum\limits_{\{k_1,\ldots,k_i\}\subset\{1,\ldots, n\}}\beta_{\{k_1,\ldots,k_i\}}E\left(X_{i:\{k_1,\ldots,k_i\} } \right)_p   \right).
\label{fmCSA}
\end{equation}
\end{itemize}
\end{theorem}
\begin{proof}
We will show (\ref{rmCS1}), (\ref{rmCSapp}), (\ref{rmCS2}) and (\ref{fmCS}). The proof of (\ref{rmCS1A}), (\ref{rmCSappA}), (\ref{rmCS2A}) and (\ref{fmCSA}) goes along the same lines. 

From (\ref{sf1}) we have
\begin{eqnarray}
P(T>t) &=&\sum_{i=1}^{n} \sum\limits_{\{k_1,\ldots,k_i\}\subset\{1,\ldots, n\}} 
\sum\limits_{j=1}^s(-1)^{j+1}\sum\limits_{\{s_1,\ldots,s_j\}\subset\{1,\ldots, s\}} \nonumber \\ && \hspace{5mm} \mathbb{I}\left(\bigcup_{l=1}^jP_{s_l}=\{k_1,\ldots,k_i\}\right) 
P(X_{1:\{k_1,\ldots,k_i\}}>t) \nonumber  \\
&=&\sum_{i=1}^{n} \sum\limits_{\{k_1,\ldots,k_i\}\subset\{1,\ldots, n\}}\alpha_{\{k_1,\ldots,k_i\}} P(X_{1:\{k_1,\ldots,k_i\}}>t).
\label{dthmCSf1}
\end{eqnarray}

\noindent 
(i) Using (\ref{fact1_formula}) and next (\ref{dthmCSf1}) we get
\begin{eqnarray*}
ET^p&=&\sum\limits_{m=0}^\infty \big((m+1)^p-m^p\big) P(T>m) \\
&=&\sum\limits_{m=0}^\infty \big((m+1)^p-m^p\big)  \sum_{i=1}^{n} \sum\limits_{\{k_1,\ldots,k_i\}\subset\{1,\ldots, n\}}
\alpha_{\{k_1,\ldots,k_i\}} P(X_{1:\{k_1,\ldots,k_i\}}>m).
\end{eqnarray*}
But from (\ref{th1d2}) and (\ref{th1d4})  we see that 
\begin{equation}
\label{prawdop_Min}
 P(X_{1:\{k_1,\ldots,k_i\}}>m)= P\left(\bigcap\limits_{l=1}^i\{X_{k_l}>m\}\right),
\end{equation}
  and  (\ref{rmCS1}) is established.

\noindent 
(ii) The error of the approximation in (\ref{rmCSapp}) is given by 
\begin{eqnarray}
Error&=&\sum\limits_{m=M_0+1}^{\infty} \big((m+1)^p-m^p\big)  \sum_{i=1}^{n} \sum\limits_{\{k_1,\ldots,k_i\}\subset\{1,\ldots, n\}}
\alpha_{\{k_1,\ldots,k_i\}} \nonumber  \\
&&\times P\left(\bigcap\limits_{l=1}^i\{X_{k_l}>m\}\right).  \label{s4error}
\end{eqnarray}
From  (\ref{dthmCSf1}) - (\ref{s4error}) it follows that 
\begin{equation}
\label{dthmCSf3}
Error\geq 0.  
\end{equation}
On the other hand, Theorem \ref{th2} shows that the assumption $EX_{j_0}^p<\infty$ guarantees that (\ref{thmCSiii}) holds and hence that the infinite series 
\begin{eqnarray*}
&&\sum\limits_{m=M_0+1}^{\infty} \big((m+1)^p-m^p\big) P(X_{1:\{k_1,\ldots,k_i\}}>m) \\
&=&\sum\limits_{m=M_0+1}^{\infty} \big((m+1)^p-m^p\big)P\left(\bigcap\limits_{l=1}^i\{X_{k_l}>m\}\right)
\end{eqnarray*}
is convergent whenever $\emptyset\neq \{k_1,\ldots,k_i\}\subset \{1,\ldots, n\}$. This allows us to change the order of summation in (\ref{s4error}) to get
\begin{eqnarray}
Error&=& \sum_{i=1}^{n} \sum\limits_{\{k_1,\ldots,k_i\}\subset\{1,\ldots, n\}}\alpha_{\{k_1,\ldots,k_i\}} \sum\limits_{m=M_0+1}^{\infty} \big((m+1)^p-m^p\big)  \nonumber  \\
&&\times P(X_{1:\{k_1,\ldots,k_i\}}>m)  \nonumber \\
&\leq& \sum_{i=1}^{n} \sum\limits_{\{k_1,\ldots,k_i\}\subset\{1,\ldots, n\}}\alpha^+_{\{k_1,\ldots,k_i\}} \sum\limits_{m=M_0+1}^{\infty} \big((m+1)^p-m^p\big)  \nonumber  \\
&&\times P(X_{1:\{k_1,\ldots,k_i\}}>m). \label{dthmCSf4}
\end{eqnarray}
 Theorem \ref{th2} implies
\begin{eqnarray}
\sum\limits_{m=M_0+1}^{\infty} \big((m+1)^p-m^p\big)P(X_{1:\{k_1,\ldots,k_i\}}>m)  && \nonumber  \\
\leq  d\left( \sum_{i=1}^{n} \sum\limits_{\{k_1,\ldots,k_i\}\subset\{1,\ldots, n\}}\alpha^+_{\{k_1,\ldots,k_i\}} \right)^{-1}, && \label{dthmCSf5}
\end{eqnarray}
provided that (\ref{M0CS}) holds. Now (\ref{dthmCSf3}) -  (\ref{dthmCSf5}) show that
$$0\leq Error \leq d,$$
which means that (\ref{rmCSapp}) gives an accuracy not worse than $d$.

\noindent 
(iii) Since $T$ is a non-negative RV we have
$$ET^p=p\int_0^{\infty}x^{p-1}P(T>x)dx$$
and from (\ref{dthmCSf1}) we see that
\begin{eqnarray*}
ET^p&=&\sum_{i=1}^{n} \sum\limits_{\{k_1,\ldots,k_i\}\subset\{1,\ldots, n\}}\alpha_{\{k_1,\ldots,k_i\}}
p\int_0^{\infty}x^{p-1}P(X_{1:\{k_1,\ldots,k_i\}}>x)dx \\
&=&\sum_{i=1}^{n} \sum\limits_{\{k_1,\ldots,k_i\}\subset\{1,\ldots, n\}}\alpha_{\{k_1,\ldots,k_i\}} E\left(X^p_{1:\{k_1,\ldots,k_i\}}  \right),
\end{eqnarray*}
which is precisely (\ref{rmCS2}).

Condition (\ref{thmCSiii}) guarantees that $ET^p$ is finite. Hence so are $ET^j$, $j=1,\ldots,p-1$. Moreover, for some real numbers $c_j$, $j=1,\ldots,p$, we have
$$x(x-1)\cdots(x-(p-1))=\sum_{j=1}^{p}c_jx^j.$$
Consequently
\begin{eqnarray*}
E(T)_p &=&\sum_{j=1}^{p}c_jE\left(T^j\right) \\
&=&\sum_{j=1}^{p}c_j \sum_{i=1}^{n} \sum\limits_{\{k_1,\ldots,k_i\}\subset\{1,\ldots, n\}}\alpha_{\{k_1,\ldots,k_i\}} E\left(X^j_{1:\{k_1,\ldots,k_i\}}  \right) \\
&=& \sum_{i=1}^{n} \sum\limits_{\{k_1,\ldots,k_i\}\subset\{1,\ldots, n\}}\alpha_{\{k_1,\ldots,k_i\}}
E\left(\sum_{j=1}^{p}c_jX^j_{1:\{k_1,\ldots,k_i\}}  \right) \\
&=& \sum_{i=1}^{n} \sum\limits_{\{k_1,\ldots,k_i\}\subset\{1,\ldots, n\}}\alpha_{\{k_1,\ldots,k_i\}}
E\left(X_{1:\{k_1,\ldots,k_i\}}  \right)_p,
\end{eqnarray*}
and (\ref{fmCS}) is proved.
\end{proof}

If the random vector $(X_1,X_2,\ldots,X_n)$ is exchangeable, then, for any $t$ and $\emptyset\neq \{k_1,\ldots,k_i\}\subset \{1,\ldots, n\}$, we have
$$P(X_{1:\{k_1,\ldots,k_i\}}>t)=P(X_{1:i}>t).$$
Hence comparing (\ref{sfexch}) with (\ref{dthmCSf1}) and using (\ref{alfy}) we obtain the following formula for the  minimal signature. For $1\leq i\leq n$,
\begin{eqnarray}
\alpha_i & = & \sum\limits_{\{k_1,\ldots,k_i\}\subset\{1,\ldots, n\}}\alpha_{\{k_1,\ldots,k_i\}} \label{alfy_exch} \\
&=&\sum_{j=1}^{s}(-1)^{j+1} \sum\limits_{\{s_1,\ldots,s_j\}\subset\{1,\ldots, s\}}\mathbb{I}\left(\left|\bigcup_{l=1}^jP_{s_l}\right|=i \right).  \nonumber
\end{eqnarray}
Similarly, for $1\leq i\leq n$,
$$\beta_i=
\sum_{j=1}^{v}(-1)^{j+1} \sum\limits_{\{s_1,\ldots,s_j\}\subset\{1,\ldots, v\}}\mathbb{I}\left(\left|\bigcup_{l=1}^jC_{s_l}\right|=i \right). $$
Minimal and maximal signatures for all coherent systems with three and four exchangeable components are given in \cite{NRS07}, those for all coherent systems with five exchangeable components are tabulated in \cite{NR10} while whose for some special cases of consecutive-$k$-out-of-$n$ systems can be found in \cite{NE07} and \cite{E10}. In general computation of minimal and/or maximal signatures of complex systems is a demanding and time consuming task. In the literature techniques of finding Samaniego signatures are considerably better developed than those for minimal and maximal signatures; see, for example \cite{DZH12}, \cite{R17}, \cite{DCX18} and \cite{YC18}. Samaniego signature of a coherent system with lifetime $T$ and $n$ exchangeable components with lifetimes $X_1,X_2,\ldots,X_n$ is a vector $\mathbf{s}=(s_1,\ldots,s_n)$ such that 
$$P(T>t)=\sum_{i=1}^ns_iP(X_{i:n}>t) \; \hbox{ for any } t$$
and $\mathbf{s}$ depends only on the structure of the system and not on the distribution of $(X_1,X_2,\ldots,X_n)$. The existence of such a vector  $\mathbf{s}$ for any coherent system with exchangeable components was proved by Navarro et. al. \cite{NSBB08} who generalized earlier results by Samaniego \cite{S85} and Navarro and Rychlik \cite{NR07}. Minimal and maximal signatures can be easily determined from the corresponding Samaniego signature. Indeed, under the assumption that $(X_1,X_2,\ldots,X_n)$ is exchangeable (\ref{recc}) can be rewritten as
\begin{eqnarray*}
P(X_{r:n}>t) &=&\sum_{j=0}^{r-1} (-1)^{r-1-j} {{n-j-1} \choose {n-r}} {n \choose n-j}P(X_{1:n-j}>t) \\
&=&\sum_{i=n-r+1}^{n} (-1)^{r-1-n+i} {{i-1} \choose {n-r}} {n \choose i}P(X_{1:i}>t).
\end{eqnarray*}
It follows that 
\begin{eqnarray*}
\sum_{r=1}^{n}s_r P(X_{r:n}>t) &=&\sum_{r=1}^{n}s_r\sum_{i=n-r+1}^{n} (-1)^{r-1-n+i} {{i-1} \choose {n-r}} {n \choose i}P(X_{1:i}>t) \\
&=&\sum_{i=1}^{n}{n \choose i}\sum_{r=n-i+1}^{n}s_r (-1)^{r-1-n+i} {{i-1} \choose {n-r}} P(X_{1:i}>t),
\end{eqnarray*}
and hence that
\begin{equation}
\label{alfyzSamaniego}
\alpha_i={n \choose i}\sum_{r=n-i+1}^{n}s_r (-1)^{r-1-n+i} {{i-1} \choose {n-r}} , \; i=1,\ldots,n.
\end{equation}
Similarly, using \cite[formula (2.2)]{BBM92} we can derive formulas for $\beta_i$, $ i=1,\ldots,n$, in terms of Samaniego signature $(s_1,\ldots,s_n)$. 

Knowledge of minimal and/or maximal signatures of a coherent system simplifies computation of moments of the lifetime of this system in the case when its components are exchangeable. In this case Theorem \ref{mCS} specializes in the following result. 
\begin{corollary}
\label{mCSexch}
Let the random vector  $(X_1,X_2,\ldots,X_n)$ be exchangeable and write $T$ for the lifetime of a coherent system with component lifetimes $X_1,X_2,\ldots,X_n$ and with minimal and maximal signatures $(\alpha_1,\ldots,\alpha_n)$ and $(\beta_1,\ldots,\beta_n)$, respectively.
\begin{itemize}
\item[(i)]  If $X_i$ takes values in the set of non-negative integers, $i=1,2,\ldots,n$, then, for $p=1,2,\ldots$,
\begin{eqnarray*}
ET^p&=&\sum\limits_{m=0}^\infty \big((m+1)^p-m^p\big)  \sum_{i=1}^{n} \alpha_{i} 
 P\left(\bigcap\limits_{l=1}^i\{X_{l}>m\}\right)   \\
&=&\sum\limits_{m=0}^\infty \big((m+1)^p-m^p\big)  \sum_{i=1}^{n} \beta_{i}
 \left\{1-P\left(\bigcap\limits_{l=1}^i\{X_{l}\leq m\}\right)\right\}.
\end{eqnarray*}
Moreover, if $EX_1^p<\infty$ then the approximate formulas
$$ET^p\approx \sum\limits_{m=0}^{M_0} \big((m+1)^p-m^p\big)  \sum_{i=1}^{n} \alpha_{i} 
 P\left(\bigcap\limits_{l=1}^i\{X_{l}>m\}\right) $$
and 
$$ET^p\approx \sum\limits_{m=0}^{\bar{M}_0} \big((m+1)^p-m^p\big) \beta_{i}
 \left\{1-P\left(\bigcap\limits_{l=1}^i\{X_{l}\leq m\}\right)\right\}$$
introduce an error not greater than $d$, provided that
$$
 \sum_{x=M_0+2}^{\infty}x^pP(X_1=x)\leq d \left(\sum_{i=1}^{n}\alpha_{i}^+\right)^{-1}
$$
and
\begin{equation*}
 \sum_{x=\bar{M}_0+2}^{\infty}x^pP(X_{1}=x)\leq \frac{d}{2^n-1} \left(\sum_{i=1}^{n} \beta_{i}^+\right)^{-1}
\end{equation*}
respectively, where
$
\alpha_{i}^+=
\left\{
\begin{array}{ll}
\alpha_{i} & \hbox{if }  \;  \alpha_{i}>0 \\
0  & \hbox{otherwise} 
\end{array}
\right.
$
and 
$
\beta_{i}^+=
\left\{
\begin{array}{ll}
\beta_{i} & \hbox{if }  \;  \beta_{i}>0 \\
0  & \hbox{otherwise} 
\end{array}
\right..
$
\item[(ii)] If $X_1,X_2,\ldots,X_n$ are non-negative RV's and for a fixed $p\in\{1,2,\ldots\}$ we have $EX_1^p<\infty$, then 
$$ET^p=\sum_{i=1}^{n} \alpha_{i} E\left(X_{1:i}^p\right)=\sum_{i=1}^{n} \beta_{i} E\left(X_{i:i}^p\right)$$
and
\begin{eqnarray}
E(T)_p&=&\sum_{i=1}^{n} \alpha_{i} E\left(X_{1:i}\right)_p \label{fmCSexch} \\
&=& \sum_{i=1}^{n} \beta_{i} E\left(X_{i:i}\right)_p.  \nonumber
\end{eqnarray}
\end{itemize}
\end{corollary}

It should be noted that (\ref{rmCS2}) was presented in a slighty diffrent form by Navarro et. al. \cite[Corollary 5.1]{NRS07}. Here we gave it for completeness. Yet in the context of discrete lifetimes of compnents, (\ref{fmCS}) proves to be more useful than (\ref{rmCS2}). For example, (\ref{factorialMin}) and (\ref{fmCS}) together with its simplified version (\ref{fmCSexch}) valid for exchangeable components yield closed-form formulas describing single moments of times to failure of coherent systems with components having MVG lifetimes. 
\begin{corollary}
\label{mCSMVG}
Let the random vector  $(X_1,X_2,\ldots,X_n)$ have the MVG distribution with parameters $\theta_I\in [0,1]$, $\emptyset\neq I\subset \{1,2,\ldots,n\}$. Write $T$ for the lifetime of a coherent system with component lifetimes $X_1,X_2,\ldots,X_n$ and with   minimal path  sets $P_1,P_2,\ldots,P_s$. Then, for $p=1,2,\ldots$,
\begin{equation}
\label{fmCSMVG}
E(T)_p=p!\sum_{i=1}^{n} \sum\limits_{\{k_1,\ldots,k_i\}\subset\{1,\ldots, n\}}\alpha_{\{k_1,\ldots,k_i\}}
\left(\frac{1}{1-\frac{\theta_{all}}{\prod\limits_{\emptyset\neq I\subset \{1,\ldots,n\}\setminus\{k_1,\ldots,k_i\}}\theta_I}}-1\right)^p.
\end{equation}
where $\alpha_{\{k_1,\ldots,k_i\}}$ and $\theta_{all}$ are given in (\ref{alfy}) and (\ref{teta.all}), respectively. 

In particular, if the component lifetimes $X_1,X_2,\ldots,X_n$ are exchangeable with the MVG distribution with parameters $\theta_I\in [0,1]$, $\emptyset\neq I\subset \{1,2,\ldots,n\}$ satisfying (\ref{cMVGexch}), 
and $(\alpha_1,\ldots,\alpha_n)$ is the minimal signature of the coherent system, then (\ref{fmCSMVG}) simplifies to
\begin{equation}
\label{fmCSMVGexch}
E(T)_p=p!\sum_{i=1}^{n} \alpha_{i}\left(\frac{1}{1-\prod\limits_{j=1}^{n} \theta_j^{{n \choose j}-{n-i \choose j}}}-1\right)^p.
\end{equation}
\end{corollary}

We illustrate Corollary \ref{mCSMVG} with the following example.
\begin{example}
Let us consider the bridge system presented in Figure \ref{bridge} and assume that the joint distribution of its component lifetimes $X_1,X_2,\ldots,X_5$ is MVG with parameters $\theta_I\in [0,1]$, $\emptyset\neq I\subset \{1,2,\ldots,5\}$.
We will find  the expectation  $ET$ and variance $Var(T)$ of  the time to failure of this system.

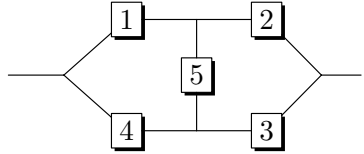
\begin{figure}[htb]
\centering
\begin{tabular}{c}
$\xymatrix@R=8pt@C=15pt{
 & & *+[F-,]{1} \ar@{-}[r] & *[o]{}  \ar@{-}[r] & *+[F-,]{2} \ar@{-}[rd] \\
*[o]{} \ar@{-}[r] & *[o]{} \ar@{-}[ru] \ar@{-}[rd] & & *+[F-,]{5} \ar@{-}[u] \ar@{-}[d] & & *[o]{} \ar@{-}[r] & \\
 & & *+[F-,]{4} \ar@{-}[r] & *[o]{} \ar@{-}[r] & *+[F-,]{3} \ar@{-}[ru]}$
\end{tabular}
\caption{\small Bridge system}\label{bridge}
\end{figure}

There are four minimal path sets of this bridge system:
$$
P_1=\{1,2\}, P_2=\{3,4\}, P_3=\{1,3,5\}, P_4=\{2,4,5\}.
$$
Moreover,
$$P_1\cup P_2=\{1,2,\ldots,5\}\setminus\{5\},$$
$$P_1\cup P_3=\{1,2,\ldots,5\}\setminus\{4\},$$
$$P_1\cup P_4=\{1,2,\ldots,5\}\setminus\{3\},$$
$$P_2\cup P_3=\{1,2,\ldots,5\}\setminus\{2\},$$
$$P_2\cup P_4=\{1,2,\ldots,5\}\setminus\{1\},$$
$$P_3\cup P_4=\{1,2,\ldots,5\}$$
and
$$\bigcup_{l=1}^jP_{i_l}=\{1,2,\ldots,5\} \; \hbox{ if } 1\leq i_1<\cdots<i_j\leq 5 \hbox{ and } j\geq 3.$$
Hence
\begin{equation}
\label{bridge_alfy}
\alpha_{\{k_1,\ldots,k_i\}}=\left\{
\begin{array}{cl}
1 & \hbox{ if } \; \{k_1,\ldots,k_i\}\in\{P_1,P_2,P_3,P_4\} \\
-1 & \hbox{ if } \; \{k_1,\ldots,k_i\}\in \left\{\{1,2,\ldots,5\}\setminus\{i\}, i=1,2,\ldots,5\right\}\\
-1+{4 \choose 3}-1=2 & \hbox{ if } \; \{k_1,\ldots,k_i\}=\{1,2,\ldots,5\}\\
0 & \hbox{ otherwise}
\end{array},
\right.
\end{equation}
by (\ref{alfy}).  From (\ref{fmCSMVG}) we get, for $p=1,2,\ldots$,
\begin{eqnarray}
E(T)_p &=& p! \left\{\left(\frac{1}{1-\frac{\theta_{all}}{\theta_{\{3\}} \theta_{\{4\}}\theta_{\{5\}} \theta_{\{3,4\}} \theta_{\{3,5\}} \theta_{\{4,5\}} \theta_{\{3,4,5\}} }}-1\right)^p \right. \nonumber \\
&&+  \left(\frac{1}{1-\frac{\theta_{all}}{\theta_{\{1\}} \theta_{\{2\}}\theta_{\{5\}} \theta_{\{1,2\}} \theta_{\{1,5\}} \theta_{\{2,5\}} \theta_{\{1,2,5\}} }}-1\right)^p   \nonumber \\
&& + \left(\frac{1}{1-\frac{\theta_{all}}{\theta_{\{2\}} \theta_{\{4\}}\theta_{\{2,4\}}  }}-1\right)^p 
+ \left(\frac{1}{1-\frac{\theta_{all}}{\theta_{\{1\}} \theta_{\{3\}}\theta_{\{1,3\}}  }}-1\right)^p \nonumber \\
&&\left. -\sum_{i=1}^{5} \left(\frac{1}{1-\frac{\theta_{all}}{\theta_{\{i\}}  }}-1\right)^p
+2\left(\frac{1}{1-\theta_{all}}-1\right)^p \right\},  \label{fmbridge}
\end{eqnarray}
where $\theta_{all}=\prod_{\emptyset\neq I\subset \{1,2,\ldots,5\}}\theta_I$. Taking $p=1$ and $p=2$ we obtain $ET$ and $E(T(T-1))$, respectively. Then we can compute $Var(T)$ using the relation $Var(T)=E(T(T-1))+ET\left(1-ET\right)$. In Table \ref{Table8} we demonstrate values of $ET$ and $Var(T)$ obtained for the following selected settings of the parameters $\theta_I$, $\emptyset\neq I\subset \{1,2,\ldots,5\}$ ($\theta_I$ which are not listed are assumed to equal $1$):
\begin{enumerate}
\item[(1)] $\theta_{\{1\}}=0.9$, $\theta_{\{3\}}=0.8$, $\theta_{\{1,4,5\}}=\theta_{\{2,3,5\}}=0.99$;
\item[(2)] $\theta_{\{1\}}=\theta_{\{2\}}=0.9$, $\theta_{\{3\}}=\theta_{\{4\}}=\theta_{\{5\}}=0.8$, $\theta_{\{1,4,5\}}=\theta_{\{2,3,5\}}=0.99$;
\item[(3)] $\theta_{\{1\}}=\theta_{\{2\}}=0.9$, $\theta_{\{3\}}=\theta_{\{4\}}=\theta_{\{5\}}=0.8$;
\item[(4)] $\theta_{\{i\}}=0.9$, $i=1,2,\ldots,5$, $\theta_{\{i,j\}}=0.95$, $i,j\in\{1,2,\ldots,5\}$ and $i\neq j$;
\item[(5)] $\theta_{\{i\}}=0.9$, $i=1,2,\ldots,5$, $\theta_{\{i,j\}}=0.95$, $i,j\in\{1,2,\ldots,5\}$ and $i\neq j$, $\theta_{\{1,2,\ldots,5\}}=0.99$.
\end{enumerate}
Clearly (3) corresponds to the case when the RV's $X_1, X_2,\ldots,X_5$ are independent, while (4) and (5) to the case when the random vector $(X_1, X_2,\ldots,X_5)$ is exchangeable. 

\begin{table}[ht]
\footnotesize
\caption{Expectation and variance of $T$ for the bridge system when $(X_1,X_2,\ldots,X_5)$ has MVG distribution with parameters $\theta_I$, $\emptyset\neq I\subset\{1,2,3,4,5\}$ (not listed parameters are assumed to be equal to 1)}
\begin{tabular}{|c||c|c||}
\hline
$\theta_I$ & E$T$ & Var$(T)$\\ \hline\hline
$\theta_{\{1\}}=0.9$, $\theta_{\{3\}}=0.8$  &49.251& 2474.938\\
$\theta_{\{1,4,5\}}=\theta_{\{2,3,5\}}=0.99$ && \\ \hline
$\theta_{\{1\}}=\theta_{\{2\}}=0.9$, $\theta_{\{3\}}=\theta_{\{4\}}=\theta_{\{5\}}=0.8$ &4.751&16.996\\
$\theta_{\{1,4,5\}}=\theta_{\{2,3,5\}}=0.99$  && \\ \hline
$\theta_{\{1\}}=\theta_{\{2\}}=0.9$, $\theta_{\{3\}}=\theta_{\{4\}}=\theta_{\{5\}}=0.8$ &5.237&20.001\\\hline
$\theta_{\{i\}}=0.9$, $i=1,\ldots,5$, $\theta_{\{i,j\}}=0.95$, $i,j\in\{1,2,\ldots,5\}$ and $i\neq j$ &2.163&4.167\\ \hline
$\theta_{\{i\}}=0.9$, $i=1,\ldots,5$, $\theta_{\{i,j\}}=0.95$, $i,j\in\{1,2,\ldots,5\}$ and $i\neq j$ &2.109&4.034\\
$\theta_{\{1,2,\ldots,5\}}=0.99$ &&\\ \hline
\end{tabular}
\label{Table8}
\end{table}

In general, in the situation when  $X_1, X_2,\ldots,X_5$  are independent, that is when $\theta_I=1$ for all $\emptyset\neq I\subset \{1,2,\ldots,5\}$ except singletons, (\ref{fmbridge}) simplifies to 
\begin{eqnarray}
E(T)_p &=& p! \left\{\left(\frac{1}{1-\theta_{\{1\}} \theta_{\{2\}}}-1\right)^p 
+\left(\frac{1}{1-\theta_{\{3\}} \theta_{\{4\}}}-1\right)^p\right.   \nonumber \\
&&+ \left(\frac{1}{1-\theta_{\{1\}} \theta_{\{3\}} \theta_{\{5\}} }-1\right)^p 
+ \left(\frac{1}{1-\theta_{\{2\}} \theta_{\{4\}} \theta_{\{5\}} }-1\right)^p \nonumber \\
&&\left. -\sum_{i=1}^{5} \left(\frac{1}{1-\frac{\theta_{all}^{ind}}{\theta_{\{i\}}  }}-1\right)^p
+2\left(\frac{1}{1-\theta_{all}^{ind}}-1\right)^p\right\}, \label{fmbridge_ind}
\end{eqnarray}
where $\theta_{all}^{ind}=\prod_{i=1}^5\theta_{\{i\}}$.

If in turn $X_1, X_2,\ldots,X_5$  are exchangeable with
\begin{equation}
\label{thety_exch_bridge}
\theta_{\{i_1,\ldots,i_s\}}=\theta_s\in[0,1] \quad \hbox{ for all } 1\leq s\leq 5 \hbox{ and } 1\leq i_1<\cdots<i_s\leq 5,
\end{equation}
where at least one $\theta_s$, $1\leq s\leq 5$, is not equal to 1, then to compute $E(T)_p$ we can use directly (\ref{fmCSMVGexch}) but first we need to find the minimal signature $(\alpha_1,\alpha_2,\ldots,\alpha_5)$ for the bridge system. We can do this using one of the following methods:
\begin{enumerate}
\item substitute (\ref{bridge_alfy}) into (\ref{alfy_exch});
\item apply (\ref{alfyzSamaniego}) together with the known fact that the Samaniego signature $(s_1,s_2,\ldots,s_5)$ for the bridge system is equal to $(0,\frac{1}{5},\frac{3}{5},\frac{1}{5},0)$;
\item find the bridge system in Table 1 of \cite{NR10} (it is in the row numbered $N=93$) and then read its minimal signature from Table 2 of \cite{NR10}.
\end{enumerate}
The result is $(\alpha_1,\alpha_2,\ldots,\alpha_5)=(0,2,2,-5,2)$. Hence  (\ref{fmCSMVGexch}) gives
\begin{eqnarray}
E(T)_p &=& p! \left\{2\left(\frac{1}{1-\frac{\theta_{all}^{exch}}{\theta_1^3\theta_2^3\theta_3}}-1\right)^p 
+2\left(\frac{1}{1-\frac{\theta_{all}^{exch}}{\theta_1^2\theta_2}}-1\right)^p\right. \nonumber \\
&&\left. -5\left(\frac{1}{1-\frac{\theta_{all}^{exch}}{\theta_1}}-1\right)^p
+2\left(\frac{1}{1-\theta_{all}^{exch}}-1\right)^p \right\}, \label{fmbridge_exch}
\end{eqnarray}
where $\theta_{all}^{exch}=\prod_{j=1}^{5}\theta_j^{{5 \choose j}}$.
Of course (\ref{fmbridge_exch}) can be also  obtained  by substituting (\ref{thety_exch_bridge}) into (\ref{fmbridge}).

In the case when $X_1,X_2,\ldots,X_5$ are IID and geometrically distributed with parameter $\pi\in(0,1)$, both (\ref{fmbridge_ind}) and (\ref{fmbridge_exch}) lead to 
\begin{eqnarray*}
E(T)_p &=& p! \left\{2\left(\frac{1}{1-(1-\pi)^2}-1\right)^p 
+2\left(\frac{1}{1-(1-\pi)^3}-1\right)^p \right. \\
&&\left. -5\left(\frac{1}{1-(1-\pi)^4}-1\right)^p 
+2\left(\frac{1}{1-(1-\pi)^5}-1\right)^p  \right\}, 
\end{eqnarray*}
by taking
$$
\theta_{\{i_1,\ldots,i_s\}}=
\left\{
\begin{array}{ll}
\theta_1=1-\pi & \hbox{ if } \{i_1,\ldots,i_s\}=\{i\}, i=1,2,\ldots,5, \\
\theta_s=1 & \hbox{ otherwise.}
\end{array}
\right.
$$
In Figure \ref{wykresETVarT} we present the expectation $ET$ and variance $Var(T)$ as functions of $\pi\in(0,0.25)$ in the case when $X_1,X_2,\ldots,X_5$ are IID and $X_i$, $i=1,2,\ldots,5$, are geometrically distributed with parameter $\pi$.

\begin{figure}
\includegraphics[width=1\textwidth]{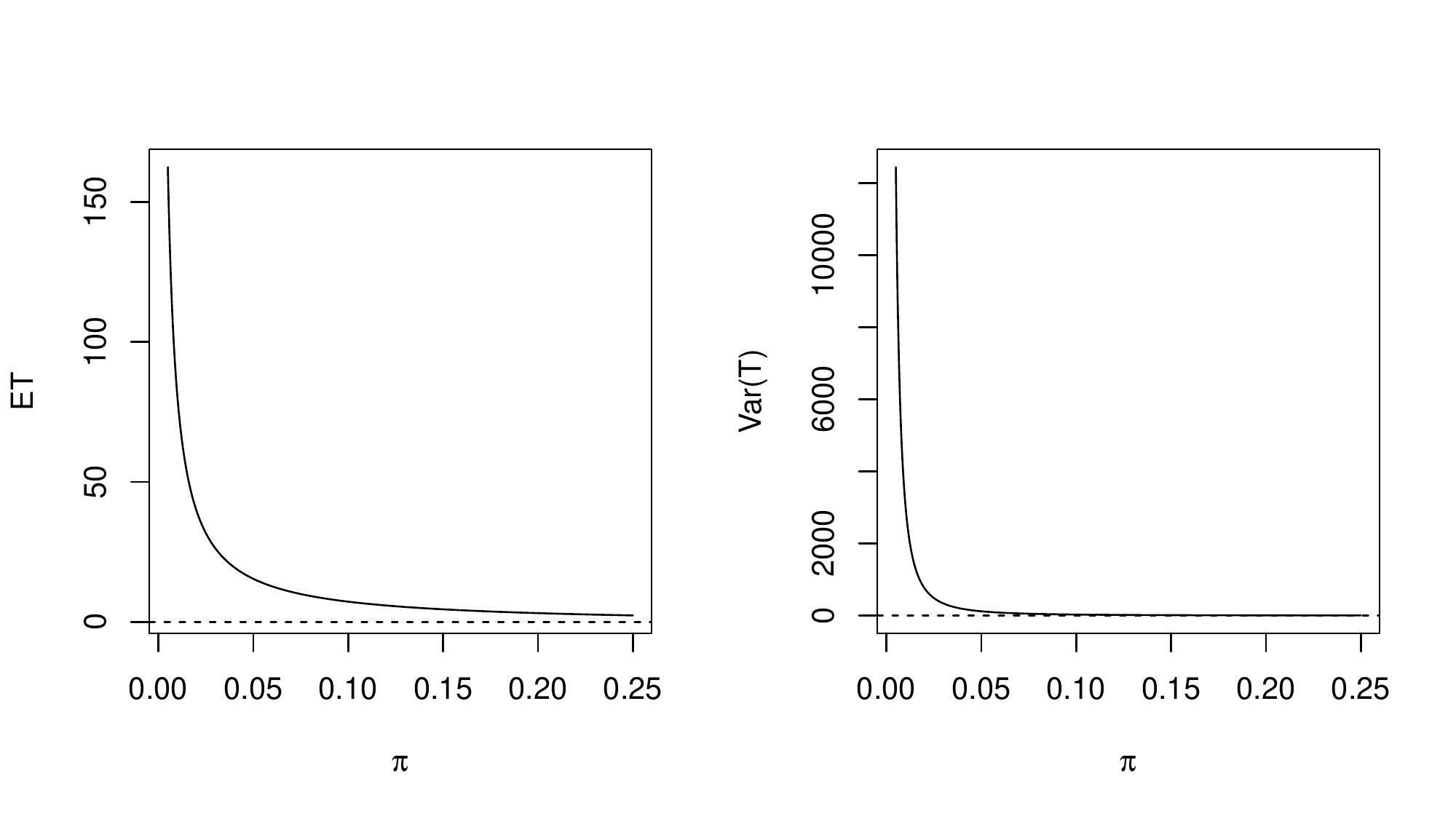}
\caption{Expectation $ET$ and variance $Var(T)$ as functions of $\pi\in(0,0.25)$ for  the bridge system when $X_1,X_2,\ldots,X_5$ are IID and $X_i\sim ge(\pi)$, $i=1,2,\ldots,5$}
\label{wykresETVarT}
\end{figure}

\end{example}

Before we present the next example let us mention a useful observation.
\begin{remark}
\label{Sec4Rem1}
If  $X_i\sim \mathrm{Pois}(\lambda_i)$, $\lambda_i>0$, $i=1,2,\ldots,n$, then the same analysis as that used in the proof of Corollary \ref{th2c1} shows that (\ref{M0CS}) is satisfied if
\begin{equation}
\label{M0Pois_system}
M_0=\left\{
\begin{array}{lcl}
p-2 & \hbox{if} &  \tilde{d}_{Pois}\leq 0 \\
F_{X_{j_0}}^{\leftarrow}(\tilde{d}_{Pois})+p-1  & \hbox{if} &  \tilde{d}_{Pois}\in(0,1)
\end{array}
\right.,
\end{equation}
where $ j_0=\argmax_{j=1,\ldots,n}\lambda_j$ and 
$$
\tilde{d}_{Pois}=1-d\, 2^{-p(p-1)/2}\lambda_{j_0}^{-p}  \left(\sum_{i=1}^{n}\sum\limits_{\{k_1,\ldots,k_i\}\subset\{1,\ldots, n\}}\alpha_{\{k_1,\ldots,k_i\}}^+\right)^{-1}.
$$
\end{remark}

\begin{example}
Let us again consider the bridge system presented in Figure~\ref{bridge} but now let  its component lifetimes $X_1,X_2,\ldots,X_5$ be independent and  $X_i\sim \mathrm{Pois}(\lambda_i)$, $\lambda_i>0$, $i=1,2,\ldots,5$.

Then using Theorem \ref{mCS} (ii), (\ref{bridge_alfy}) and Remark \ref{Sec4Rem1} we see that to compute $ET^p$, $p=1,2,\ldots$, with an error not greater than $d>0$ we can apply the approximate formula
\begin{eqnarray}
ET^p&\approx&\sum\limits_{m=0}^{M_0} \big((m+1)^p-m^p\big)  \sum_{i=1}^{n} \sum\limits_{\{k_1,\ldots,k_i\}\subset\{1,\ldots, n\}}
\alpha_{\{k_1,\ldots,k_i\}}  \prod_{l=1}^iP\left(X_{k_l}>m\right) \nonumber \\
&=&\sum\limits_{m=0}^{M_0} \big((m+1)^p-m^p\big) 
\bigg\{P(X_1>m)P(X_2>m)+P(X_3>m)P(X_4>m) \nonumber \\
&&+\left(\prod_{i=1}^5P(X_i>m)\right)\bigg(\frac{1}{P(X_2>m)P(X_4>m)}+\frac{1}{P(X_1>m)P(X_3>m)} \nonumber \\
&& -\sum_{j=1}^5\frac{1}{P(X_j>m)}+2\bigg)
\bigg\}, \label{bridge_Pois}
\end{eqnarray}
where  $M_0$ is given by (\ref{M0Pois_system}) with $ j_0=\argmax_{j=1,\ldots,5}\lambda_j$ and $\tilde{d}_{Pois}=1-d\, 2^{-p(p-1)/2}\lambda_{j_0}^{-p}6^{-1}$.

Moreover, if $X_1,X_2,\ldots,X_5$ are not only independent but also identically distributed, then (\ref{bridge_Pois}) reduces to 
\begin{eqnarray*}
ET^p&\approx&2\sum\limits_{m=0}^{M_0} \big((m+1)^p-m^p\big) \left(P(X_1>m)\right)^2 \bigg\{1+P(X_1>m) \\
&& -\frac{5}{2}\left(P(X_1>m)\right)^2+\left(P(X_1>m)\right)^3\bigg\}.
\end{eqnarray*}

For illustrative purposes we fixed the level of accuracy $d=0.0005$. Then in Table \ref{Table9}, for some selected values of $\lambda_1, \lambda_2,\ldots,\lambda_5$, we demonstrate expectations $ET$ and second raw moments $ET^2$ for the bridge system when $X_1,X_2,\ldots,X_5$ are  independent and $X_i\sim \mathrm{Pois}(\lambda_i)$, $i=1,2,\ldots,5$. 
Additionally in brackets we provide the corresponding values of $M_0$, that is  the numbers of terms in the sum sufficient to obtain the desired accuracy.
Next, under the stronger assumption that $X_1,X_2,\ldots,X_5$ are IID with $X_i\sim \mathrm{Pois}(\lambda)$, in Figure \ref{wykresETET2}, we present $ET$ and  $ET^2$ as functions of $\lambda\in(0,100)$. It is interesting that   the function $ET(\lambda)$ describing the dependence of $ET$ on $\lambda\in(1,100)$ is such that $ET(\lambda)\approx \lambda$. Moreover, numerical calculations show that this property holds also for $\lambda\geq100$, and $|ET(\lambda)-\lambda|$ decreases as $\lambda$ increases.

\bigskip
\begin{table}[ht]
\footnotesize
\caption{Expectation and second raw moment of $T$ for the bridge system when $X_1,X_2,\ldots,X_5$ are independent and $X_i\sim Pois(\lambda_i)$, $i=1,2,\ldots,5$}
\begin{tabular}{|c||c|c||}
\hline
$\lambda_i$ & E$T$ & E$T^2$\\ \hline\hline
$\lambda_i=1$, $i=1,2,\ldots,5$ &0.877&1.246\\
                                &(6)   &(8)\\\hline
$\lambda_i=i$, $i=1,2,\ldots,5$ &2.728&8.935\\
                                &(17)&(19)\\\hline
$\lambda_i=6-i$, $i=1,2,\ldots,5$ &3.458&13.980\\
                                  &(17)    &(19)\\\hline
$\lambda_1=\lambda_2=10$, $\lambda_3=\lambda_4=20$, $\lambda_5=50$ &17.600&321.251\\
                                                    &(86)&(95)\\\hline
$\lambda_1=\lambda_4=20$, $\lambda_2=50$, $\lambda_3=\lambda_5=10$&20.103&422.855\\
                 &(86)&(95)\\\hline
\end{tabular}
\label{Table9}
\end{table}

\begin{figure}
\includegraphics[width=1\textwidth]{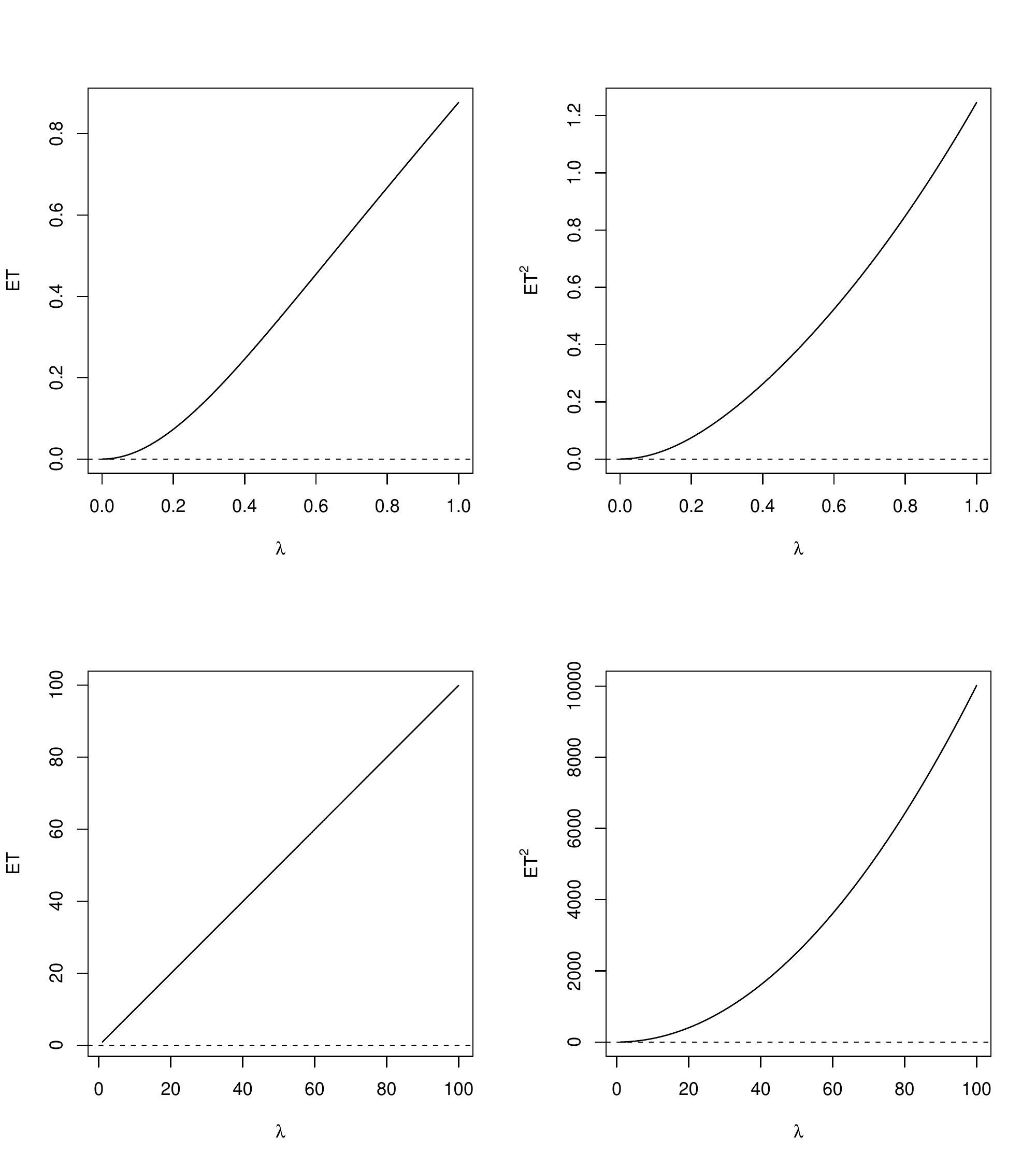}
\caption{Expectation $ET$ and second raw moment $ET^2$ as functions of $\lambda\in(0,100)$ for  the bridge system when $X_1,X_2,\ldots,X_5$ are IID and $X_i\sim \mathrm{Pois}(\lambda)$, $i=1,2,\ldots,5$}
\label{wykresETET2}
\end{figure}

\end{example}

\section*{Acknowledgments}
The authors wish to thank Prof. Jorge Navarro for suggesting extension of formulas describing moments of lifetimes of $k$-out-of-$n$ systems to those valid for any coherent system. 

A. D. and A. G. acknowledge financial support for this research from Warsaw University of Technology under Grant no. 504/03910/1120 and Polish National Science Center under Grant no. 2015/19/B/ST1/03100, respectively.


\end{document}